\documentclass[10pt]{article}
\usepackage{amsthm}
\usepackage{fullpage}
\usepackage{amssymb, amsmath}
\usepackage{graphicx}
\usepackage{titlesec}
\usepackage{enumitem}
\usepackage{comment}
\usepackage{relsize}
\usepackage{comment}
\usepackage{xcolor}
\usepackage{mathtools}
\usepackage[font=small]{caption}

\usepackage[backend=bibtex,giveninits=true,doi=false,url=false,isbn=false,style=trad-plain,hyperref,date=year]{biblatex}
\usepackage[hidelinks,colorlinks=true,linkcolor=blue,citecolor=blue]{hyperref}

\newtheorem{thm}{Theorem}
\newtheorem{prop}[thm]{Proposition}
\newtheorem{lem}[thm]{Lemma}
\newtheorem{cor}[thm]{Corollary}
\newtheorem{conj}[thm]{Conjecture}
\newtheorem{claim}{Claim}
\theoremstyle{definition}
\newtheorem{prob}[thm]{Problem}

\numberwithin{thm}{section}

\newcommand{\cB}{\mathcal{B}}
\newcommand{\cC}{\mathcal{C}}
\newcommand{\cD}{\mathcal{D}}
\newcommand{\cE}{\mathcal{E}}
\newcommand{\cF}{\mathcal{F}}
\newcommand{\cG}{\mathcal{G}}

\newcommand{\cI}{\mathcal{I}}

\newcommand{\cK}{\mathcal{K}}

\newcommand{\cM}{\mathcal{M}}
\newcommand{\cN}{\mathcal{N}}

\newcommand{\cR}{\mathcal{R}}
\newcommand{\cS}{\mathcal{S}}
\newcommand{\cT}{\mathcal{T}}

\newcommand{\cV}{\mathcal{V}}
\newcommand{\cW}{\mathcal{W}}

\newcommand{\recongraph}{\mathbf{RG}}

\newcommand{\colorfulcomplex}{\mathbf{Col}}
\newcommand{\intersectioncomplex}{\mathbf{Int}}
\newcommand{\ordercomplex}{\mathbf{Ord}}
\newcommand{\colorfulnerve}{\mathbf{CN}}
\newcommand{\nerve}{\mathbf{N}}
\newcommand{\convex}{\mathrm{conv}}
\newcommand{\colorfulsubdivision}{f}
\newcommand{\bilinearfunction}{g}

\title{Hall's theorem for reconfigurations and higher dimensional topological connectedness}
\author{Ronen Wdowinski\thanks{Institute of Discrete Mathematics, Graz University of Technology, Steyrergasse 30, 8010 Graz, Austria. Email: wdowinski@math.tugraz.at.}}
\date{\today}

\begin{document}

\maketitle

\begin{abstract}
One widely applied sufficient condition for the existence of a colorful simplex in a vertex-colored simplicial complex is a topological extension of Hall's transversal theorem due to Aharoni, Haxell, and Meshulam. We prove a similar topological Hall theorem that provides a sufficient condition for being able to transform any colorful simplex into any other through a sequence of one-vertex swaps while always maintaining a colorful simplex, meaning that the associated reconfiguration graph is connected. In fact, we prove a generalized topological Hall theorem about the homological connectedness of the space of colorful simplices, as well as a matroidal generalization of this result. We deduce sufficient conditions for reconfiguration graphs to be connected for various combinatorial structures of interest such as independent transversals in graphs, matchings in bipartite hypergraphs, and intersections of matroids. In particular, we give an alternative proof of a maximum degree condition for independent transversal reconfigurability by Buys, Kang, and Ozeki. We also deduce tight reconfiguration versions of the colorful Helly, colorful Carath\'eodory, and Tverberg theorems from discrete geometry, confirming a conjecture of Oliveros, Rold{\'a}n, Sober{\'o}n, and Torres.
\end{abstract}

\section{Introduction}
We use the following notation throughout this paper: Given subsets $V_1, \ldots, V_n$ of a finite set $V$ and an index subset $I \subseteq [n] \coloneqq \{1,\ldots,n\}$, we denote $V_I \coloneqq \bigcup_{i \in I} V_i$.

Combinatorial reconfiguration problems have attracted a lot of attention in the fields of graph theory and theoretical computer science. See \cite{mynhardt2019reconfiguration, nishimura2018introduction} for recent surveys on the topic, including reconfigurations of proper vertex-colorings, independent sets, dominating sets, and satisfiability sets. Reconfiguration problems are usually concerned with finding step-by-step transformations from one feasible solution to another feasible solution in such a way that every intermediate step is also a feasible solution. The solution space to a given problem is typically represented as a reconfiguration graph, whose vertices represent feasible solutions and whose edges represent valid single-step moves. A particular problem of interest is to decide when a given reconfiguration graph is connected. For example, it would imply that associated Markov chains on the solution space have unique stationary distributions, which lend themselves to the possibility of efficient approximate sampling and counting algorithms if one could also show that the Markov chains are rapidly mixing. The goal of this paper is to demonstrate how topological methods developed around Hall's transversal theorem can be adapted to showing that reconfiguration graphs are connected in various natural combinatorial settings, and how these results extend to higher dimensional connectedness.

Our main theorems are about colorful simplices in vertex-colored simplicial complexes. An abstract simplicial complex (or \textit{complex}) $\cC$ is a collection of subsets of some finite ground set $V$ that is closed under taking subsets. Each set in $\cC$ called a \textit{simplex} or \textit{face}, and $V(\cC) \coloneqq \bigcup \cC$ is called the set of \textit{vertices} of $\cC$. Given a complex $\cC$ and a partition $\cV$ of its vertex set, a \textit{colorful simplex} of $(\cC, \cV)$ is a simplex in $\cC$ that contains exactly one vertex from each class in $\cV$. This terminology comes from viewing $\cV$ as the color classes of an associated vertex-coloring $\lambda : V(\cC) \rightarrow [n]$. We define a reconfiguration graph $\recongraph(\cC, \cV)$ as follows. The vertex set of the graph is the collection of all colorful simplices of $(\cC, \cV)$, and two colorful simplices of $(\cC, \cV)$ are joined by an edge in the graph if their union is a simplex in $\cC$ of size $|\cV|+1$, i.e., if they are faces of a common simplex with one more vertex. Note that this adjacency condition is stronger than simply requiring that the colorful simplices differ in one color class, but it is crucial for applying topological methods. 

One general and powerful tool for showing the existence of colorful simplices in vertex-colored simplicial complexes is a topological extension of Hall's theorem (Theorem \ref{thm:topological-hall}) which was originally noted by Aharoni (see \cite{meshulam2001clique}) and proven by Meshulam \cite{meshulam2003domination, meshulam2001clique}. An earlier version was proven implicitly by Aharoni and Haxell \cite{aharoni2000hall}. One of our main theorems (Theorem \ref{thm:topological-reconfiguration-graph}) is that a surprisingly straightforward variation of this topological Hall theorem, namely an excess version of it, provides a sufficient condition for the reconfiguration graph $\recongraph(\cC, \cV)$ to be connected. Thus, we uncover a fundamental link, from the topological point of view, between the topics of existence and reconfiguration of combinatorial structures. Roughly speaking, they correspond to different levels of topological connectedness. We further extend this line of thought by also proving a generalized topological Hall theorem (Theorem \ref{thm:topological-reconfiguration-complex}) about the higher dimensional topological connectedness of the space of colorful simplices, in a homological setting. We study this space of colorful simplices in the form of a simplicial complex that we call the \textit{colorful complex}. It resembles the well-known homomorphism complex of graphs \cite{babson2006complexes, babson2007proof}, which itself has also been studied topologically for reconfiguration applications \cite{dochtermann2009hom, dochtermann2023homomorphism, wrochna2020homomorphism}. We also extend our new topological Hall theorems to matroidal settings (Theorems \ref{thm:complex-matroid-reconfiguration} and \ref{thm:complex-matroid-connectedness}), in the spirit of Aharoni and Berger \cite{aharoni2006intersection}. We describe applications of our reconfiguration theorems to problems in graph theory and discrete geometry, roughly showing that certain excess versions of some existence results imply connected reconfiguration graphs, with the conditions often being tight. These parallel reconfiguration results in other areas, such as Jerrum's result \cite{jerrum1995very} that the reconfiguration graph on proper $k$-colorings of a graph $G$ with maximum degree $\Delta$ is connected when $k \ge \Delta+2$.

Our original motivation comes from work of Buys, Kang, and Ozeki \cite{buys2025reconfiguration}, who introduced and studied reconfigurations of independent transversals in graphs. Given a graph $G$ and partition $\cV = \{V_1, \ldots, V_n\}$ of its vertices, an \textit{independent transversal} of $(G, \cV)$ is an independent set of $G$ consisting of one vertex from each class $V_i$ of $\cV$. In other words, it is a colorful simplex of $(\cI(G), \cV)$, where $\cI(G)$ is the collection of all independent sets of $G$. A celebrated theorem of Haxell \cite{haxell2001note} states that if $G$ has maximum degree $\Delta$ and $|V_i| \ge 2\Delta$ for all $i$, then $(G, \cV)$ always has an independent transversal. Buys, Kang, and Ozeki \cite{buys2025reconfiguration} proved the following extension of Haxell's theorem to the reconfiguration setting.

\begin{thm}[\cite{buys2025reconfiguration}] \label{thm:BKO}
Let $G$ be a graph with maximum degree $\Delta$, and let $\cV = \{V_1, \ldots, V_n\}$ be a partition of $V(G)$ such that $|V_i| \ge 2\Delta$ for all $i$. If $G[V_I]$ is not the disjoint union of $|I|$ copies of the complete bipartite graph $K_{\Delta,\Delta}$, for all nonempty $I \subseteq [n]$, then $\recongraph(\cI(G), \cV)$ is connected.
\end{thm} 

In particular, Theorem \ref{thm:BKO} implies that $\recongraph(\cI(G), \cV)$ is connected whenever $|V_i| \ge 2\Delta+1$ for all $i$. The proof method of \cite{buys2025reconfiguration} was combinatorial (adapting ideas from \cite{graf2020finding}), as was the original proof of Haxell's theorem \cite{haxell1995condition, haxell2001note}. On the other hand, Haxell's theorem and many variations of it have also been proven using the topological Hall theorem mentioned above. We use our reconfiguration variation of the topological Hall theorem to give an alternative proof of Theorem \ref{thm:BKO}. We also give a simplified combinatorial proof of the corollary that $|V_i| \ge 2\Delta+1$ for all $i$ implies a connected reconfiguration graph.
Other graph theory applications of our extension of the topological Hall theorem include reconfigurations of rainbow matchings in hypergraphs, of matchings in bipartite hypergraphs, and of list colorings in graphs. 

In the direction of discrete geometry, we prove tight reconfiguration versions of the colorful Helly, colorful Carath\'eodory, and Tverberg theorems. These are deduced from a reconfiguration analogue of the topological colorful Helly theorem due to Kalai and Meshulam \cite{kalai2005topological}. In particular, our reconfiguration version of Tverberg's theorem confirms a conjecture of Oliveros, Rold{\'a}n, Sober{\'o}n, and Torres \cite[Conjecture 2]{oliveros2025tverberg}. To describe our result, we are given a finite point set $X \subset \mathbb{R}^d$ and an integer $r \ge 1$. The reconfiguration graph $\recongraph_{\mathrm{Tv}}(X, r)$ has vertex set consisting of all ordered partitions $(X_1, \ldots, X_r)$ of $X$ into $r$ classes whose convex hulls have nonempty intersection ($\bigcap_{i=1}^r \convex(X_i) \neq \emptyset$), so-called \textit{ordered Tverberg $r$-partitions}. Two ordered Tverberg $r$-partitions are joined by an edge in the graph if they differ by the placement of a single point $x$ among the $r$ classes, and excluding $x$ from either of the two $r$-partitions results in an ordered Tverberg $r$-partition of $X - \{x\}$. Tverberg's theorem \cite{tverberg1966generalization} states that $\recongraph_{\mathrm{Tv}}(X, r)$ is nonempty whenever $|X| \ge (d+1)(r - 1)+1$. We show the following.

\begin{thm} \label{thm:Tverberg-reconfiguration-1}
Let $X$ be a finite set of points in $\mathbb{R}^d$, and let $r \ge 1$ be an integer. If $|X| \ge (d+1)(r-1)+2$, then $\recongraph_{\mathrm{Tv}}(X, r)$ is connected.
\end{thm}

Oliveros, Rold{\'a}n, Sober{\'o}n, and Torres \cite{oliveros2025tverberg} considered a similar reconfiguration graph on $X \subset \mathbb{R}^d$ and $r \ge 1$, but with their vertices being \textit{unordered Tverberg $r$-partitions} $\{X_1, \ldots, X_r\}$, and adjacency not requiring that the exclusion of an element $x$ in the two unordered Tverberg $r$-partitions also be an unordered Tverberg $r$-partition. They conjectured the existence of an integer $k = k(d)$, depending only on the dimension $d$, such that their reconfiguration graph is connected whenever $|X| \ge (d+1)(r-1) +1+ k(d)$, and they showed that $|X| \ge 3(d+1)(r-1) + 2$ is sufficient. Theorem \ref{thm:Tverberg-reconfiguration-1} implies that we can take $k(d) = 1$ for all dimensions $d$, the best possible (see \cite[Example 1]{oliveros2025tverberg}). A higher dimensional homological connectedness version of Theorem \ref{thm:Tverberg-reconfiguration-1} is given by Theorem \ref{thm:tverberg-homological-connectedness}.

This paper is organized as follows. In Section \ref{sec:main-theorem}, we state our main theorems and introduce the topological tools used in our proofs and applications. In Section \ref{sec:proofs}, we give both our homotopical and homological proofs of our reconfiguration topological Hall theorem. In Section \ref{sec:higher}, we prove our generalized topological Hall theorem about the topological connectedness of the space of colorful simplices. In Section \ref{sec:intersections}, we generalize our topological Hall theorems to matroidal settings. In Section \ref{sec:combinatorial-applications}, we describe graph theory applications about reconfigurations of independent transversals, rainbow matchings, bipartite hypergraph matchings, and list colorings. In Section \ref{sec:geometric-applications}, we describe discrete geometry applications about reconfiguration versions of the colorful Helly, colorful Carath\'eodory, and Tverberg theorems, as well as their higher dimensional connectedness generalizations. We conclude with a few general problems and further lines of investigations.

\section{Main theorems and topological tools} \label{sec:main-theorem}

In this section, we state our main theorems and collect relevant topological tools for the proofs and applications.

\subsection{Main theorems} \label{sec:main-theorems}

Recall that an \textit{abstract simplicial complex} (or \textit{complex}) $\cC$ is a nonempty collection of subsets $\sigma$ (called \textit{simplices} or \textit{faces}) of some finite ground set $V$ such that if $\sigma \in \cC$ and $\tau \subseteq \sigma$, then $\tau \in \cC$. We distinguish the ground set $V$ from the \textit{vertex set} $V(\cC) \coloneqq \bigcup \cC$, although if a ground set is not specified then we implicitly assume $V = V(\cC)$. The \textit{dimension} of simplex $\sigma$ of $\cC$ is $d \coloneqq |\sigma| - 1$, and we call $\sigma$ a \textit{$d$-simplex}. A \textit{geometric simplex} $\sigma$ is the convex hull of affinely independent points in $\mathbb{R}^d$ (called its \textit{extreme points} or \textit{vertices}), and a \textit{face} of $\sigma$ is the convex hull of some finite subset of the vertices of $\sigma$. A \textit{geometric simplicial complex} $\cK$ is a collection of geometric simplices in $\mathbb{R}^d$ such that (1) if $\sigma \in \cK$ and $\tau$ is a face of $\sigma$ then $\tau \in \cK$, and (2) if $\sigma, \sigma' \in \cK$ then $\sigma \cap \sigma'$ is a face of both $\sigma$ and $\sigma'$. The union of all simplices of $\cK$ is called the \textit{polyhedron} of $\cK$ and denoted $\lVert \cK \rVert$. A standard fact is that every complex $\cC$ can be realized as a geometric simplicial complex $\cK$, meaning that the simplices in $\cC$ are precisely the vertex sets of the geometric simplices in $\cK$. The polyhedron $\lVert \cK \rVert$ of a geometric realization of $\cC$ is unique up to homeomorphism, and we denote the associated topological space by $\lVert \cC \rVert$.

A complex $\cC$ is said to be \textit{homotopically $k$-connected} if, for every integer $-1 \le j \le k$, every continuous map $f : S^j \rightarrow \lVert \cC \rVert$ can be extended to a continuous map $\tilde{f} : B^{j+1} \rightarrow \lVert \cC \rVert$ (where $B^{j+1}$ denotes a closed $(j+1)$-dimensional ball and $S^j$ denotes its boundary $j$-dimensional sphere). A complex $\cC$ is said to be \textit{homologically $k$-connected} if, for every integer $-1 \le j \le k$, the $j$th reduced simplicial homology group $\widetilde{H}_j(\cC) = \widetilde{H}_j(\cC; \mathbb{Q})$ with rational coefficients vanishes. For intuition, being (homotopically or homologically) $(-1)$-connected means being nonempty, being $0$-connected means being path-connected, and being homotopically $1$-connected is the same as being simply-connected. Since our simplicial complexes are finite, being homotopically $\infty$-connected is the same as being contractible, whereas being homologically $\infty$-connected is commonly known as being $\mathbb{Q}$-acyclic. For many of our results, though not all, the choice of rational coefficients for homological connectedness could be replaced by coefficients in any field or even in any abelian group.

Following notation of Aharoni and Berger \cite{aharoni2006intersection}, the parameter $\eta_{\pi}(\cC)$ (respectively, $\eta_H(\cC)$) is two plus the maximum integer $k$ for which complex $\cC$ is homotopically (respectively, homologically) $k$-connected. In short,
\begin{align*}
	\eta_{\pi}(\cC) &\coloneqq \max \{k : \pi_j(\cC) = 0 \text{ for all } {-}2 \le j \le k\} + 2, \\
	\eta_H(\cC) &\coloneqq \max \{k : \widetilde{H}_j(\cC) = 0 \text{ for all } {-}2 \le j \le k\} + 2,
\end{align*}
where $\pi_j(\cC)$ denotes the $j$th homotopy group of $\cC$ with respect to an arbitrary base point, and again $\widetilde{H}_j(\cC)$ denotes the $j$th reduced homology group of $\cC$ with rational coefficients. By convention, we always have $\pi_{-2}(\cC) = \widetilde{H}_{-2}(\cC) = 0$, whereas $\pi_{-1}(\cC) = \widetilde{H}_{-1}(\cC) = 0$ is taken to mean that $\cC$ is nonempty. In general, we have $\eta_H(\cC) \ge \eta_{\pi}(\cC)$, and Hurewicz's theorem implies that equality holds whenever $\eta_{\pi}(\cC) \ge 3$. We use the notation $\eta$ whenever a result applies to both $\eta_{\pi}$ and $\eta_H$. For a subset $X$ of the ground set $V$ of $\mathcal{C}$, the induced subcomplex of $\cC$ on ground set $X$ is
\begin{align*}
	\cC[X] \coloneqq \{\sigma \in \cC : \sigma \subseteq X\}.
\end{align*}

The following is the standard topological Hall theorem about the existence of colorful simplices in vertex-colored simplicial complexes. It was proven by Meshulam \cite{meshulam2003domination, meshulam2001clique} in the stronger homological setting, with the homotopical version proven earlier implicitly by Aharoni and Haxell \cite{aharoni2000hall}. The result broadly generalizes the hard direction of Hall's theorem \cite{hall1935representation} as well as Rado's theorem \cite{rado1942theorem} from matroid theory.

\begin{thm}[\cite{meshulam2003domination}] \label{thm:topological-hall}
Let $\cC$ be a complex, let $\cV = \{V_1, \ldots, V_n\}$ be a partition of $V(\cC)$, and let $\eta \in \{\eta_{\pi}, \eta_H\}$. If
\begin{align*}
	\eta (\cC [ V_I ] ) \ge |I| \hspace{25pt} \text{for all } I \subseteq [n],
\end{align*}
then $(\cC, \cV)$ has a colorful simplex. 
\end{thm}

There is also a well-known and easily-derived \textit{deficiency} version of Theorem \ref{thm:topological-hall} (see \cite{aharoni2001ryser, haxell2019topological}). This states that for any fixed integer $d \ge 0$, if 
\begin{align*}
	\eta (\cC [ V_I ] ) \ge |I| - d \hspace{25pt} \text{for all } I \subseteq [n],
\end{align*}
then $(\cC, \cV)$ has a \textit{partial colorful simplex} of size $n - d$ (that is, a simplex on $n - d$ vertices containing at most one vertex of each color). This deficiency version played a crucial role, for example, in Aharoni's proof \cite{aharoni2001ryser} of Ryser's conjecture \cite{henderson1971permutation} for the case of $3$-uniform hypergraphs. Our main theorem is one step in the other direction, stating that an \textit{excess} version of Theorem \ref{thm:topological-hall} provides a sufficient condition for the reconfiguration graph on colorful simplices to be connected. 

\begin{thm} \label{thm:topological-reconfiguration-graph}
Let $\cC$ be a complex, let $\cV = \{V_1, \ldots, V_n\}$ be a partition of $V(\cC)$, and let $\eta \in \{\eta_{\pi}, \eta_H\}$. If
\begin{align*}
	\eta (\cC [ V_I ] ) \ge |I| + 1 \hspace{25pt} \text{for all nonempty } I \subseteq [n],
\end{align*}
then $\recongraph(\cC, \cV)$ is connected.
\end{thm}

Unlike the deficiency version, Theorem \ref{thm:topological-reconfiguration-graph} does not appear to be easily reduced to Theorem \ref{thm:topological-hall}, although the proof ideas are similar. While the statement of Theorem \ref{thm:topological-reconfiguration-graph} is stronger and more widely applicable for homological connectedness $\eta_H$, the proof is more intuitive for homotopical connectedness $\eta_{\pi}$ and also relates better to classical topological results. Specifically, our homotopical proof of Theorem \ref{thm:topological-reconfiguration-graph} is a reduction to a reconfiguration variation of Sperner's classical triangulation lemma \cite{sperner1928neuer} which can be shown to be equivalent to a parameterized extension of Brouwer's fixed point theorem due to Browder \cite{browder1960continuity}. This theorem of Browder states that for every continuous map $f : \Delta^n \times [0,1] \rightarrow \Delta^n$, where $\Delta^n$ is an embedded $n$-dimensional simplex, the parameterized fixed point set $\{(x,t) \in \Delta^n \times [0,1] : f(x,t) = x\}$ contains a connected component that intersects both $\Delta^n \times \{0\}$ and $\Delta^n \times \{1\}$. Such a connected component is not necessarily guaranteed to be path-connected (see \cite[Example 2]{solan2023browder}). An equivalent parameterized extension of the KKM theorem was proven by Herings and Talman \cite{herings1998intersection}. See \cite{eaves1972homotopies, garcia1975fixed, herings2001variational, kulpa2000parametric, solan2023browder, talman2012parameterized} for some other related or equivalent results.

Our homological proof of Theorem \ref{thm:topological-reconfiguration-graph} is inspired by the homological proof of Theorem \ref{thm:topological-hall} given in \cite[Proposition 2.6]{deloera2019discrete} (also found in \cite{cho2025colorful}). In fact, we later extend our homological proof and give a generalized topological Hall theorem about the homological connectedness of the space of colorful simplices. For a complex $\cC$ and a partition $\cV = \{V_1, \ldots, V_n\}$ of $V(\cC)$, we define the \textit{colorful complex} $\colorfulcomplex(\cC, \cV)$ to be the abstract simplicial complex whose vertex set is the collection of all simplices $\sigma \in \cC$ that span the classes of $\cV$ (i.e., that contain at least one vertex of every color), and where $\sigma_1, \ldots, \sigma_{\ell}$ form a simplex of this complex whenever $\sigma_1 \subset \cdots \subset \sigma_{\ell}$. The complex $\colorfulcomplex(\cC, \cV)$, an order complex, is the barycentric subdivision of a natural polyhedral complex whose $1$-skeleton is the reconfiguration graph $\recongraph(\cC, \cV)$. It is similar in construction to the well-known \textit{homomorphism complex} between two graphs $G$ and $H$, a complex that was studied by Babson and Kozlov \cite{babson2006complexes, babson2007proof} (see also \cite{dochtermann2009hom}) in connection to Lov\'asz's \cite{lovasz1978kneser} classical work on the chromatic numbers of Kneser graphs. We prove the following.

\begin{thm} \label{thm:topological-reconfiguration-complex}
Let $\cC$ be a complex, let $\cV = \{V_1, \ldots, V_n\}$ be a partition of $V(\cC)$, and let $m \ge 0$ be an integer. If
\begin{align*}
	\eta_H (\cC [ V_I ] ) \ge |I| + m \hspace{25pt} \text{for all nonempty } I \subseteq [n],
\end{align*}
then $\eta_H(\colorfulcomplex(\cC, \cV)) \ge m+1$.
\end{thm}

Theorem \ref{thm:topological-reconfiguration-complex} provides a sufficient condition for a different reconfiguration graph to be connected, namely the $(m-1)$-dimensional up-down walk on the complex $\colorfulcomplex(\cC, \cV)$. This reconfiguration graph has vertex set consisting of all $(m-1)$-simplices of $\colorfulcomplex(\cC, \cV)$, and two $(m-1)$-simplices are joined by an edge if they are faces of a common $m$-simplex of $\colorfulcomplex(\cC, \cV)$. Such a reconfigurability result is obtained using the simplicial Hodge theorem relating the vanishing of cohomology over the reals to a positive spectral gap of the reconfiguration graph (see \cite{mukherjee2016random}).

In Section \ref{sec:intersections}, we further generalize Theorems \ref{thm:topological-reconfiguration-graph} and \ref{thm:topological-reconfiguration-complex} to matroidal settings, with the main results being Theorems \ref{thm:complex-matroid-reconfiguration} and \ref{thm:complex-matroid-connectedness}. These follow similar lines to Aharoni and Berger's \cite{aharoni2006intersection} matroidal generalization of the topological Hall theorem, but with an extra ingredient from posets. In particular, Theorem \ref{thm:complex-matroid-connectedness} implies an analogous deficiency version of Theorem \ref{thm:topological-reconfiguration-complex} (see Theorem \ref{thm:topological-hall-deficiency}). For combinatorial concreteness and brevity, the combinatorial applications described in this paper focus on reconfiguration results coming from Theorem \ref{thm:topological-reconfiguration-graph} and \ref{thm:complex-matroid-reconfiguration}, although higher dimensional results may be deduced from Theorems \ref{thm:topological-reconfiguration-complex} and \ref{thm:complex-matroid-connectedness}.

\subsection{Further topological notions}

Our proofs and applications will use basic notions from algebraic topology which we briefly review below. See \cite{hatcher2002algebraic, munkres2018elements, spanier2012algebraic} for standard references on algebraic topology.

For the homotopical proof of Theorem \ref{thm:topological-reconfiguration-graph}, we use the following notions. A \textit{triangulation} of a topological space $X$ is a complex $\cT$ whose geometric realization is homeomorphic to $X$. If $X$ is a connected, triangulable, $d$-dimensional manifold, such as a ball or sphere, then any triangulation $\cT$ of $X$ has the following three properties:
\begin{itemize}
	\item[(1)] (\textit{pure}) Every simplex in $\cT$ is contained in some $d$-simplex in $\cT$.
	\item[(2)] (\textit{nonbranching}) Every $(d - 1)$-simplex is a face of exactly one or two $d$-simplices, if $d \ge 2$.
	\item[(3)] (\textit{strongly connected}) For every pair $\sigma, \tau$ of $d$-simplices in $\cT$, there is a sequence of $d$-simplices $\sigma = \sigma_0, \sigma_1, \ldots, \sigma_{k} = \tau$ such that $\sigma_{i-1} \cap \sigma_{i}$ is a $(d-1)$-simplex for all $1 \le i \le k$.
\end{itemize}
Complexes satisfying properties (1), (2), and (3) are usually called $d$-dimensional \textit{pseudomanifolds}. Properties (1) and (2) will implicitly be used in our homotopical proof of Theorem \ref{thm:topological-reconfiguration-graph}.

For two complexes $\cC$ and $\cD$, a \textit{simplicial map} $f : \cC \rightarrow \cD$ is a function from the ground set of $\cC$ to the ground set of $\cD$ with the property that $f(\sigma) \in \cD$ for all $\sigma \in \cC$. The following proposition, a consequence of the simplicial approximation theorem, is a useful discrete characterization of homotopical connectedness (see \cite[Proposition 2.8]{szabo2006extremal} or \cite[Proposition 5.2.33]{narins2015extremal}). 

\begin{prop}[\cite{szabo2006extremal}] \label{prop:connectedness}
A complex $\cC$ is homotopically $k$-connected if and only if for every integer $-1 \le j \le k$ and every simplicial map $f : \cS^{j} \rightarrow \cC$, where $\cS^{j}$ is a triangulation of a $j$-sphere, there exists a triangulation $\cB^{j+1}$ of a $(j+1)$-ball with $\cS^{j}$ as its boundary and a simplicial map $\tilde{f} : \cB^{j+1} \rightarrow \cC$ that extends $f$.
\end{prop}

For our homological proof of Theorem \ref{thm:topological-reconfiguration-graph}, we follow standard terminology from simplicial homology. Fix a commutative ring $\cR$ (with unit). For an integer $p \ge 0$, an \textit{oriented $p$-simplex} $[v_0, \ldots, v_p]$ is an ordering of the vertices of a $p$-simplex $\{v_0, \ldots, v_p\} \in \cC$, under the equivalence relation $[v_{\pi(0)}, \ldots, v_{\pi(p)}] = \text{sign}(\pi) \cdot [v_0, \ldots, v_{p}]$ for any permutation $\pi$ of $\{0,1, \ldots, p\}$. The \textit{$p$-dimensional chain group} $C_p(\cC) = C_p(\cC; \cR)$ is the free $\cR$-module generated by the oriented $p$-simplices of $\cC$. The elements of $C_p(\cC)$ are called \textit{$p$-chains}. The boundary operator $\partial = \partial_p : C_{p}(\cC) \rightarrow C_{p-1}(\cC)$ is defined on oriented simplices by
\begin{align*}
	\partial_p([v_0, \ldots, v_p]) \coloneqq \sum_{i = 0}^p (-1)^i [v_0, \ldots, \hat{v}_i, \ldots, v_p],
\end{align*}
where $[v_0, \ldots, \hat{v}_i, \ldots, v_p]$ is the oriented simplex obtained by removing vertex $v_i$ from $[v_0, \ldots, v_p]$. The map $\partial_p$ extends to $C_{p}(\cC)$ by linearity. In the setting of reduced homology, which is the case for us, the chain group $C_{-1}(\cC)$ is taken to be $\cR$, and $\partial_{0} \coloneqq \epsilon$ is the \textit{augmentation map}, defined by $\epsilon(a_0 [v_0] + \cdots + a_k [v_k]) \coloneqq a_0 + \cdots + a_k$, where $a_0, \ldots, a_k \in \cR$. The submodule $Z_p(\cC)$ of $C_p(\cC)$ is defined to be the kernel of $\partial_p$, whose elements are called \textit{$p$-cycles}. The submodule $B_p(\cC)$ of $C_p(\cC)$ is defined to be the image of $\partial_{p+1}$, whose elements are called \textit{$p$-boundaries}. We have $\partial_p \circ \partial_{p+1} = 0$, and thus $B_p(\cC)$ is a submodule of $Z_p(\cC)$. The quotient module
\begin{align*}
	\widetilde{H}_p(\cC) \coloneqq Z_p(\cC) / B_p(\cC)
\end{align*}
is called the \textit{$p$th reduced homology group} of $\cC$ (with coefficients in $\cR$). 

Every simplicial map $f : \cC \rightarrow \cD$ induces a homomorphism $f_{\#} = (f_{\#})_p : C_p(\cC) \rightarrow C_p(\cD)$ for all integers $p \ge -1$, defined on oriented simplices by $f_{\#}([v_0, \ldots, v_p]) \coloneqq [f(v_0), \ldots, f(v_p)]$ if $f(v_0), \ldots, f(v_p)$ are distinct, and $f_{\#}([v_0, \ldots, v_p]) = 0$ otherwise. Here, $(f_{\#})_{-1}$ denotes the identity map on $\cR$. A standard fact is that $f_{\#}$ is an \textit{augmentation-preserving chain map}, meaning it commutes with the boundary operator (including the augmentation map $\partial_0 = \epsilon$), i.e., $\partial_p \circ (f_{\#})_{p} = (f_{\#})_{p-1} \circ \partial_p$.

\subsection{Some complexes and their topological connectedness} \label{sec:lower-bounds}

Here, we review some useful results on the topological connectedness of complexes that are relevant for some applications of Theorem \ref{thm:topological-reconfiguration-graph}. See \cite{aharoni2006intersection, aharoni2024coloring, haxell2016independent, haxell2019topological} for more detailed surveys on this topic.

The \textit{join} of two complexes $\cC$ and $\cD$, assuming their ground sets are disjoint, is the complex $\cC \ast \cD \coloneqq \{\sigma \cup \tau : \sigma \in \cC, \tau \in \cD\}$. The following describes the topological connectedness of joins.

\begin{prop}[\cite{aharoni2006intersection, aharoni2015cooperative}] \label{prop:joins}
For any two complexes $\cC$ and $\cD$ on disjoint ground sets, we have $\eta_{\pi}(\cC \ast \cD) \ge \eta_{\pi}(\cC) + \eta_{\pi}(\cD)$ and $\eta_H(\cC \ast \cD) = \eta_H(\cC) + \eta_H(\cD)$.
\end{prop}

The main complexes for which we apply Theorem \ref{thm:topological-reconfiguration-graph} are \textit{independence complexes}, \textit{matching complexes}, \textit{matroids}, and \textit{nerves}. The \textit{independence complex} $\cI(G)$ of a graph $G$ is the collection of independent sets of $G$ (i.e., vertex subsets of $G$ not inducing any edges). The \textit{matching complex} $\cM(G)$ of a multi-hypergraph $G$ is the collection of all matchings of $G$ (i.e., edge subsets of $G$ that are pairwise disjoint). The matching complex of $G$ is the same as the independence complex of the line graph of $G$. \textit{Matroids} are described below, while \textit{nerves} are described in the applications when they are relevant. To apply topological Hall theorems, one needs general lower bounds on the topological connectedness of the complexes of interest. 

When it comes to the independence complex $\cI(G)$ of a graph $G$, two distinct general useful tools for lower bounding its topological connectedness are a recursive lower bound of Meshulam \cite{meshulam2003domination} and an eigenvalue bound of Aharoni, Berger, and Meshulam \cite{aharoni2005eigenvalues} (the latter only applicable for homological connectedness with rational coefficients). Among many other lower bounds, these have been used to establish various domination-type lower bounds for $\eta(\cI(G))$ \cite{haxell2019topological, meshulam2003domination, meshulam2001clique}, including the following two which have also been established using more direct triangulation arguments \cite{aharoni2006intersection, aharoni2002triangulated, aharoni2000hall}. For terminology, a vertex subset $X$ of $G$ is said to \textit{strongly dominate} a vertex subset $Y$ if every vertex in $Y$ is adjacent to some vertex in $X$. (The more typical notion of graph domination would only require that every vertex in $Y - X$ be adjacent to some vertex of $X$.) 

\begin{thm}[\cite{aharoni2006intersection, aharoni2000hall, meshulam2001clique}] \label{thm:domination-connectedness}
For any graph $G$ and $\eta \in \{\eta_{\pi}, \eta_H\}$, we have
\begin{itemize}
	\item[(a)] $\eta(\cI(G)) \ge \tilde{\gamma}(G)/2$, where $\tilde{\gamma}(G)$ is the minimum size of a strongly dominating set of $V(G)$;
	\item[(b)] $\eta(\cI(G)) \ge i\gamma(G)$, where $i\gamma(G)$ is the minimum integer $\ell$ such that every independent set of $G$ is strongly dominated by a vertex set of size at most $\ell$.  
\end{itemize}
\end{thm}

For lower bounds on the topological connectedness of matching complexes $\cM(G)$ of multi-hypergraphs $G$, the \textit{matching number} $\nu(G)$ of multi-hypergraph $G$ is the cardinality of a largest matching of $G$. A \textit{fractional matching} of $G$ is a real vector $x \in \mathbb{R}^{E(G)}$ such that $\sum_{e \ni v} x_e \le 1$ for all $v \in V(G)$, and $x_e \ge 0$ for all $e \in E(G)$. The \textit{fractional matching number} $\nu^{\ast}(G)$ is the maximum of $\sum_{e \in E(G)} x_e$ among all fractional matchings $x$ of $G$. Note that $\nu^{\ast}(G) \ge \nu(G)$. The lower bounds in the theorem below were shown in \cite{aharoni2000hall} and \cite{aharoni2005eigenvalues}, respectively. The first one is an easy consequence of Theorem \ref{thm:domination-connectedness}(b).

\begin{thm}[\cite{aharoni2005eigenvalues, aharoni2000hall}] \label{thm:matching-number}
For any $r$-graph $G$, we have $\eta(\cM(G)) \ge \frac{\nu(G)}{r}$ and $\eta_H(\cM(G)) \ge \frac{\nu^\ast(G)}{r}$.
\end{thm}

Finally, we describe relevant terminology and facts from matroid theory that we use in Section \ref{sec:intersections}. See \cite{oxley2006matroid} for more background on matroid theory. A \textit{matroid} $\cM$ on ground set $V$ (or, more accurately, its independence complex) is an abstract simplicial complex on $V$ that satisfies the following \textit{augmentation property}: if $A, B \in \cM$ and $|A| < |B|$, then there exists $x \in B$ such that $A \cup \{x\} \in \cM$. The sets in $\cM$ are usually called \textit{independent}, otherwise they are \textit{dependent}. All maximal independent sets of a matroid $\cM$ have the same size, and this size is called the \textit{rank} of $\cM$ and denoted by $r(\cM)$. A maximal independent set is called a \textit{basis}, while a minimal dependent set is called a \textit{circuit}. For any subset $X \subseteq V$, the induced subcomplex $\cM[X] \coloneqq \{A \in \cM : A \subseteq X\}$ is again a matroid, and so is the contraction $\cM / X \coloneqq \{A \subseteq V - X : A \cup B \in \cM\}$, where $B$ is any fixed basis of $\cM[X]$. For an integer $k \ge 0$, the \textit{$k$-truncation} of $\cM$ is the matroid $\cM_{\le k} \coloneqq \{A \in \cM : |A| \le k \}$. A \textit{flat} of $\cM$ is a subset $F \subseteq V(\cM)$ such that $r(\cM[F \cup \{x\}]) > r(\cM[F])$ for all $x \in V - F$. The \textit{dual matroid} of $\cM$ is the matroid  $\cM^{\ast} \coloneqq \{A \subseteq V : r(\cM[V - A]) = r(\cM)\}$, which has rank $r(\cM^{\ast}) = |V| - r(\cM)$. More generally, $r(\cM^{\ast}[X]) = r(\cM[V - X]) + |X| - r(\cM)$. We have $(\cM^\ast)^\ast = \cM$. A \textit{loop} of $\cM$ is singleton subset of $V$ that is dependent in $\cM$, whereas a \textit{coloop} of $\cM$ is a loop of $\cM^{\ast}$. 

There are many classes of matroids, and the most relevant one for us implicit throughout this paper is the class of \textit{partition matroids}. Here, given a partition $\cV = \{V_1, \ldots, V_n\}$ of a ground set $V$, the associated partition matroid is $\cM_{\cV} \coloneqq \{ A \subseteq V : |A \cap V_i| \le 1 \text{ for all } i\}$. The following well-known result gives the topological connectedness of matroids.

\begin{prop}[\cite{bjorner1992homology}] \label{prop:matroid-connectedness}
If $\cM$ is a matroid and $\eta \in \{\eta_{\pi}, \eta_H\}$, then $\eta(\cM) = r(\cM)$ if $\cM$ has no coloop, and $\eta(\cM) = \infty$ otherwise.
\end{prop}


\section{Proof of topological Hall theorem for reconfigurations} \label{sec:proofs}

In this section, we prove Theorem \ref{thm:topological-reconfiguration-graph}, our reconfiguration variation of the topological Hall theorem, for both of the topological connectedness parameters $\eta_{\pi}$ and $\eta_H$. To conform with standard topological notation, particularly when it comes to homology, in this section we take the vertex partition of our complex $\cC$ to be $\cV \coloneqq \{V_0, \ldots, V_n\}$ (the indexing starts at $0$).

\subsection{A Sperner-type lemma} \label{sec:sperner}

In this subsection, we state and prove a reconfiguration variation of Sperner's lemma \cite{sperner1928neuer}. This is the first step in our homotopical proof of Theorem \ref{thm:topological-reconfiguration-graph}, and although it is not needed for our homological proof, it helps to motivate its idea. Our specific Sperner-type lemma does not appear to have been written explicitly before, but basically the same ideas have appeared in various forms in some previous works (e.g., \cite{garcia1975fixed, herings2001variational, kulpa2000parametric, talman2012parameterized}). Our proof also somewhat resembles Gale's proof \cite{gale1979game} that the game of Hex cannot end in a draw. As mentioned in Section \ref{sec:main-theorems}, all of these results, including ours, are closely related to Browder's \cite{browder1960continuity} parameterized extension of Brouwer's fixed point theorem, as well as Herings and Talman's \cite{herings1998intersection} equivalent parameterized extension of the KKM theorem.

First we recall Sperner's lemma \cite{sperner1928neuer} in its classical form. Let $\Delta^{n}$ denote an embedded $n$-dimensional simplex, say $\Delta^n$ is the convex hull of the affinely independent point set $\{x_0, \ldots, x_n\}$. Let $\cT$ be any finite triangulation of $\Delta^{n}$. A vertex-coloring $\lambda : V(\cT) \rightarrow \{0,1,\ldots,n\}$ is said to be a \textit{Sperner coloring} of $\cT$ if each of the extreme points $x_i$ of $\Delta^n$ is assigned the color $i$, and every other vertex $v$ of $\cT$ is assigned one of the colors present among the extreme points of the minimal face containing $v$. We implicitly associate the vertex-coloring $\lambda : V(\cT) \rightarrow \{0,1,\ldots,n\}$ with the partition of $V(\cT)$ given by its color classes, i.e., $\cV \coloneqq \{\lambda^{-1}(i) : i \in \{0,1,\ldots,n\}\}$.

\begin{lem}[Sperner's lemma \cite{sperner1928neuer}] \label{lem:Sperner}
Let $\cT$ be a triangulation of an $n$-dimensional simplex $\Delta^{n}$, and let $\lambda : V(\cT) \rightarrow \{0,1,\ldots,n\}$ be a Sperner coloring of $\cT$. Then $(\cT, \lambda)$ has an odd number of colorful simplices.
\end{lem}

Proceeding to our new setting, an $(n+1)$-dimensional \textit{simplicial prism} is a polytope of the form $P^{n+1} \coloneqq \Delta^{n} \times [0,1]$. Again assume that the simplex $\Delta^{n}$ is the convex hull of the affinely independent point set $\{x_0, \ldots, x_{n}\}$. Then, for all integers $0 \le k \le n+1$, every $k$-dimensional (polyhedral) face of $P^{n+1}$ is the convex hull of one of the following types of collections of extreme points:
\begin{itemize}
	\item[(a)] $\{(x_i,0) : i \in I \}$ or $\{(x_i,1) : i \in I\}$ for some nonempty $I \subseteq \{0,1,\ldots,n\}$ with $|I| = k+1$;
	\item[(b)] $\{(x_i,0), (x_i,1) : i \in I\}$ for some nonempty $I \subseteq \{0,1,\ldots,n\}$ with $|I| = k$.
\end{itemize}
The faces with extreme points of type (a) are called \textit{base faces}, and the faces with extreme points of type (b) are called \textit{lateral faces}. All faces of $P^{n+1}$ of dimension $n$ are called \textit{facets}. The \textit{supporting face} of a point $v \in P^{n+1}$, denoted $\text{supp}(v)$, is the inclusion-wise minimal face of $P^{n+1}$ that contains $v$. For a face $F$ of $P^{n+1}$, we denote by $I(F)$ the subset $I \subseteq \{0,1,\ldots,n\}$ that defines the extreme points of $F$ as in (a) or (b):
\begin{align*}
	I(F) \coloneqq \{i \in \{0,1,\ldots,n\} : (x_i, 0) \in F \text{ or } (x_i, 1) \in F \}.
\end{align*}

Let $\cT$ be a triangulation of an $(n+1)$-dimensional simplicial prism $P^{n+1}$. We say that a vertex-coloring $\lambda : V(\cT) \rightarrow \{0,1,\ldots,n\}$ is an \textit{R-Sperner coloring} of $\cT$ if $\lambda(v) \in I(\text{supp}(v))$ for every vertex $v \in V(\cT)$. In other words, it is an assignment of a color to each vertex of $\cT$ such that
\begin{itemize}
	\item[(i)] the extreme points $(x_i,0), (x_i,1) \in V(\cT)$ are assigned the color $i$, for all $i \in \{0,1,\ldots,n\}$, and
	\item[(ii)] every other vertex $v \in V(\cT)$ is assigned one of the colors that is present among the extreme points of the supporting face of $v$.
\end{itemize}
A useful observation is that an R-Sperner coloring of $\cT$ restricts to an ordinary Sperner coloring on both of the base facets $\Delta^{n} \times \{0\}$ and $\Delta^{n} \times \{1\}$ of $P^{n+1}$. The following is our Sperner-type lemma, with the statement and proof illustrated for the case $n = 1$ in Figure \ref{fig:sperner}.

\begin{figure}
\begin{center}
	\leavevmode
	\includegraphics[scale=1]{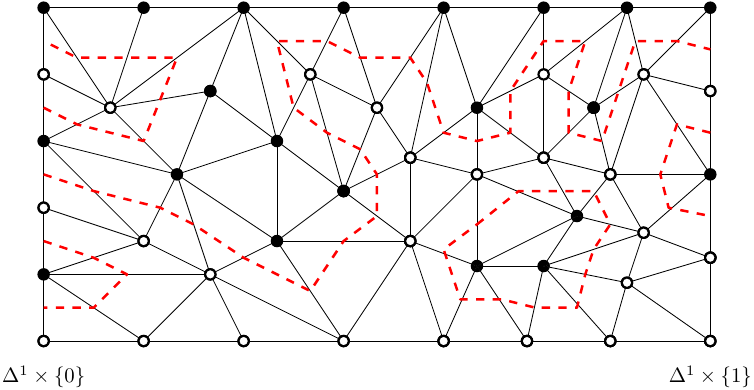}
\end{center}
\caption{An illustration of our Sperner-type lemma (Lemma \ref{lem:Sperner}) on the prism $P^2 = \Delta^1 \times [0,1]$. The associated graph $G$ is represented by dashed segments which connect adjacent colorful simplices. There is an odd number of paths that connect a colorful simplex in $\Delta^1 \times \{0\}$ to a colorful simplex in $\Delta^1 \times \{1\}$.}
\label{fig:sperner}
\end{figure}

\begin{lem} \label{lem:Sperner-reconfiguration}
Let $\cT$ be a triangulation of an $(n+1)$-dimensional simplicial prism $P^{n+1} \coloneqq \Delta^{n} \times [0,1]$, and let $\lambda : V(\cT) \rightarrow \{0,1,\ldots,n\}$ be an R-Sperner coloring of $\cT$. Then there is an odd number of sequences $S_0, S_1, \ldots, S_N$ of colorful simplices $S_i$ of $(\cT, \lambda)$ with the properties that $S_0 \subseteq \Delta^{n} \times \{0\}$, $S_N \subseteq \Delta^{n} \times \{1\}$, and $S_{j - 1}, S_j$ are faces of a common $(n+1)$-simplex for all $1 \le j \le N$.
\end{lem}

\begin{proof}
Construct a graph $G$ whose vertices are the colorful simplices of $(\cT, \lambda)$, and two colorful simplices are joined by an edge of $G$ whenever they are faces of a common $(n+1)$-simplex of $\cT$. Every colorful simplex is contained in either one or two $(n+1)$-simplices of $\cT$, and those contained in exactly one $(n+1)$-simplex are precisely those that lie in a facet of $P^{n+1}$. In addition, every $(n+1)$-simplex of $\cT$ contains either zero or two colorful simplices as faces. Thus, the vertices of the graph $G$ all have degree $1$ or $2$. Hence, $G$ is a collection of paths and cycles. Now, by Sperner's lemma (Lemma \ref{lem:Sperner}), the set $C_0$ of colorful simplices contained in the base facet $\Delta^{n} \times \{0\}$ has odd cardinality, and the same is true for the set $C_1$ of colorful simplices contained in the base facet $\Delta^{n} \times \{1\}$. On the other hand, by construction there are no colorful simplices contained in any of the lateral facets of $P^{n+1}$. Thus, $C_0 \cup C_1$ is the set of all degree $1$ vertices of $G$. The vertices of $C_0$ that are not path-connected to any vertex of $C_1$ in $G$ come in pairs that are path-connected to each other. This leaves an odd number of vertices of $C_0$ that are each path-connected to distinct vertices of $C_1$ in $G$, which proves the desired statement.
\end{proof}

\subsection{Homotopical connectedness} \label{sec:homotopical-proof}

We now proceed to our homotopical proof of Theorem \ref{thm:topological-reconfiguration-graph}. This is done by an appropriate reduction to Lemma \ref{lem:Sperner-reconfiguration}, and it is modeled after proofs of the topological Hall theorem via Sperner's lemma (e.g., \cite{aharoni2000hall,haxell2019topological,szabo2006extremal}).

\begin{proof}[Proof of Theorem \ref{thm:topological-reconfiguration-graph} for $\eta = \eta_{\pi}$]
Let $S \coloneqq \{u_0, \ldots, u_n\}$ and $T \coloneqq \{v_0, \ldots, v_n\}$ be the vertex sets of two colorful simplices of $(\cC,\cV)$, where $u_i, v_i \in V_i$ for all $i \in \{0,1,\ldots,n\}$. We wish to show the existence of a reconfiguration sequence from $S$ to $T$ consisting of colorful simplices of $(\cC,\cV)$. Let $\Delta^{n}$ denote an embedded $n$-dimensional simplex which is the convex hull of point set $\{x_0, \ldots, x_n\}$, and consider the $(n+1)$-dimensional simplicial prism $P \coloneqq \Delta^{n} \times [0,1]$. Inductively, for $0 \le k \le n+1$, we construct a triangulation $\cT_k$ of the $k$-skeleton $P^{(k)}$ of $P$ and a simplicial map $f_k : \cT_k \rightarrow \cC$ satisfying the following three properties:
\begin{itemize}
	\item[(1)] $\cT_k$ includes all base faces of $P$ up to dimension $k$;
	\item[(2)] $f_k(x_i,0) = u_i$ and $f_k(x_i,1) = v_i$ for all $i \in \{0,1\ldots,n\}$; and
	\item[(3)] the vertex-coloring $\lambda_k : V(\cT_k) \rightarrow \{0,1,\ldots,n\}$, defined by $\lambda_k(x) \coloneqq i$ whenever $f_k(x) \in V_i$, satisfies property (i) of an R-Sperner coloring, as well as property (ii) for all faces of $P$ up to dimension $k$.
\end{itemize}
	
We start with $k = 0$. The $0$-skeleton $\cT_0 \coloneqq P^{(0)}$ is just the set of extreme points of $P$, and we define the simplicial map $f_0 : \cT_0 \rightarrow \cC$ by setting $f_0(x_i,0) \coloneqq u_i$ and $f_0(x_i,1) \coloneqq v_i$ for all $i \in \{0,1,\ldots,n\}$. The properties (1), (2), (3) are satisfied. Assuming that we have defined $\cT_{k-1}$ and $\lambda_{k-1}$ satisfying the properties (1), (2), (3), we now define $\cT_{k}$ and $f_{k}$. First, for every $k$-dimensional base face $F$ of $P$, we simply include $F$ in $\cT_k$. This is allowed because by (1) the boundary of $F$ is included in $\cT_{k-1}$. No new vertices were added in $F$, so taking $f_k$ to be the same as $f_0$ on $V(F)$, we get a simplicial map from $F$ to $\cC$ satisfying (2). Now consider a $k$-dimensional lateral face $F$ of $P$, which has $|I(F)| = k$. The boundary of $F$ is contained in $P^{(k-1)}$, so that $\cT_{k-1}$ contains a triangulation $\cS_F$ of the boundary of $F$. Note that $\cS_F$ is a triangulated $(k-1)$-dimensional sphere. Since $\lambda_{k-1}$ is R-Sperner on the facets of $F$, we have $f_{k-1}(V(\cS_F)) \subseteq V_{I(F)}$. Thus, $f_{k-1}$ restricts to a simplicial map $g_{F} : \cS_F \rightarrow \cC[V_{I(F)}]$. Since 
\begin{align*}
	\eta_{\pi}(\cC[V_{I(F)}]) \ge |I(F)|+1 = k+1
\end{align*}
by our theorem assumption, there exists a triangulation $\cB_F$ of $F$ whose boundary is $\cS_F$, and a simplicial map $\tilde{g}_{F} : \cB_F \rightarrow \cC[V_{I(F)}]$ extending $g_F$. We include the $k$-skeleton of $\cB_F$ in $\cT_k$, and we define $f_{k}$ to be $\tilde{g}_F$ on $V(\cB_F)$. Doing this for all $k$-dimensional faces $F$ of $P$, we obtain a triangulation of $\cT_k$ of $P^{(k)}$ extending $\cT_{k-1}$ and a simplicial map $f_k : \cT_k \rightarrow \cC$ extending $f_{k-1}$. The induced vertex-coloring $\lambda_k : V(\cT_k) \rightarrow \{0,1,\ldots,n\}$ defined in property (3) is R-Sperner by construction. This completes the construction of $\cT_k$ and $f_k$.
	
Now we apply Lemma \ref{lem:Sperner-reconfiguration} to the triangulation $\cT \coloneqq \cT_{n+1}$ and the induced $R$-Sperner coloring $\lambda \coloneqq \lambda_{n+1}$ of $\cT$ defined in property (3). This yields a sequence $S_0, S_1, \ldots, S_N$ of colorful simplices of $(\cT,\lambda)$ such that $S_0 = \Delta^{n} \times \{0\}$ and $S_N = \Delta^{n} \times \{1\}$ (using property (1)), and $S_{j-1}, S_j$ are faces of a common $(n+1)$-simplex of $\cT$ for all $1 \le j \le N$. We also have that $f(V(S_0)) = S$ and $f(V(S_N)) = T$ (using property (2)). Since each $S_j$ is a colorful simplex of $(\cT,\lambda)$, each of the images $f(V(S_j))$ forms a colorful simplex of $(\cC,\cV)$. Moreover, $f(V(S_{j-1}))$ and $f(V(S_{j}))$ are either the same colorful simplex, or they are faces of the common $n$-simplex $f(V(S_{j-1}) \cup V(S_{j}))$ in $\cC$. Thus, after removing redundant colorful simplices, the sequence of images $S = f(V(S_0)), f(V(S_1)), \ldots, f(V(S_{N-1})), f(V(S_N)) = T$ gives a reconfiguration from $S$ to $T$ consisting of colorful simplices of $(\cC,\cV)$. This finishes the proof.
\end{proof}

\subsection{Homological connectedness} \label{sec:homological-proof}

In this subsection, we give our homological proof of Theorem \ref{thm:topological-reconfiguration-graph}. Our proof takes inspiration from the homological proof of the topological Hall theorem given in \cite[Proposition 2.6]{deloera2019discrete}, where the authors describe a simplified argument using $\mathbb{Z}_2$-coefficients. This approach was also more recently used in \cite{cho2025colorful} to give a common generalization of Sperner's lemma and the topological Hall theorem in a homology setting. It is a different approach from Meshulam's original proof \cite{meshulam2003domination}, which uses nerves and implies a weaker reconfiguration result (Theorem \ref{thm:colorful-simplex-nerve} when $m = 1$).

Our proof uses reduced homology with coefficients in a fixed commutative ring $\cR$, which is subsequently specialized to $\cR = \mathbb{Q}$ so as to phrase in terms of $\eta_H$. In our current setting, the abstract $n$-dimensional simplex $\Delta^n$ is taken to have vertex set $\{0,1,\ldots,n\}$, and for an oriented simplex $I = [i_0, \ldots, i_p] \in C_p(\Delta^n)$ we sometimes abuse notation and treat it like a subset, e.g., writing $i_0 \in I$ or $V_I \coloneqq \bigcup_{j=0}^p V_{i_j}$ or $|I| = p+1$. For an $n$-chain $c \in C_n(\cC)$, the \textit{colorful support} of $c$ is defined to be the $n$-chain obtained from $c$ by keeping only the terms which are oriented colorful simplices of $(\cC, \cV)$. The following is our key lemma.

\begin{lem} \label{lem:chain-homotopy}
Let $S = [ u_0, \ldots, u_n ]$ and $T = [ v_0, \ldots, v_n ]$ be two oriented colorful simplices of $(\cC, \cV)$, with $u_i, v_i \in V_i$ for all $i$. If
\begin{align*}
	\widetilde{H}_{|I| - 1}(\cC[V_I]) = 0 \hspace{25pt} \text{for all nonempty } I \subseteq \{0,1,\ldots,n\},
\end{align*}
then there exists a chain $K \in C_{n+1}(\cC)$ whose boundary $\partial K$ has colorful support $T - S$.
\end{lem}

\begin{proof}
Let $f, g : \Delta^n \rightarrow \cC$ denote the simplicial maps given by $f(i) \coloneqq u_i$ and $g(i) \coloneqq v_i$ for all $i \in \{0, 1, \ldots, n\}$. These induce augmentation-preserving chain maps $f_{\#}, g_{\#}: C_p(\Delta^n) \rightarrow C_p(\cC)$, given by $f_{\#}([ i_0, \ldots, i_p ]) \coloneqq [ u_{i_0}, \ldots, u_{i_p} ]$ and $g_{\#}([ i_0, \ldots, i_p ]) \coloneqq [ v_{i_0}, \ldots, v_{i_p} ]$, and with $f_{\#}(a) = g_{\#}(a) = a$ for all $a \in \cR$. Our goal is to construct, in increasing dimensions $p \ge 1$, a homomorphism $D = D_p : C_{p}(\Delta^n) \to C_{p+1}(\cC)$ satisfying the following two properties for all oriented simplices $I \in C_p(\Delta^n)$:
\begin{itemize}
	\item[(1)] $D(I) \in C_{|I|}(\cC[V_I])$, 
	\item[(2)] $\partial (D(I)) + D(\partial I) = g_{\#}(I) - f_{\#}(I)$.
\end{itemize}
Property (1) is a Sperner-type condition, whereas property (2) states that $D$ is a \textit{chain homotopy} between $f_{\#}$ and $g_{\#}$. Intuitively, we want $D(I)$ to be supported only on vertices in $V_I$ and to have the geometric structure of a prism between base facets $f_{\#}(I)$ and $g_{\#}(I)$ with lateral facets $D(\partial I)$.

We start by setting $D(a) = 0$ for all $a \in \cR$. Assume that we have defined a homomorphism $D$ satisfying properties (1) and (2) on all oriented simplices up to dimension $p - 1$, where $p \ge 0$. For an oriented simplex $I \in C_p(\Delta^{n})$, we define the chain
\begin{align*}
	c \coloneqq g_{\#}(I) - f_{\#}(I) - D(\partial I),
\end{align*}
which lies in $C_{|I|-1}(\cC[V_I])$. We claim that $\partial c = 0$. If $p = 0$, say $I = [ i ]$, we have $c = [ v_i ] - [ u_i ]$, so that $\partial c = \epsilon ([ v_i ] - [ u_i ]) = 0$. And if $p \ge 1$, by induction and the fact that $f_{\#}$ and $g_{\#}$ are chain maps, we have
\begin{align*}
	\partial c &= \partial (g_{\#}(I)) - \partial (f_{\#}(I))  - \partial (D(\partial I)) \\
	&= \partial (g_{\#}(I)) - \partial (f_{\#}(I))  - (g_{\#}(\partial I) - f_{\#}(\partial I) - D(\partial^2 I)) = 0,
\end{align*}
as required. Since $\widetilde{H}_{|I| - 1}(\cC[V_I]) = 0$ by assumption, having $\partial c = 0$ implies that $c$ is the boundary of some chain $c' \in C_{|I|}(\cC[V_I])$, and we set $D(I) \coloneqq c'$. Having defined $D$ on all oriented simplices $I \in C_p(\Delta^n)$, we extend $D$ to all chains in $C_p(\Delta^n)$ by linearity. Such a $D$ will satisfy properties (1) and (2) for all oriented simplices $I \in C_q(\Delta^n)$ of dimension $q \le p$. Doing this procedure up to dimension $p = n$, this completes our construction of $D$.

Now with $I^\ast \coloneqq [0,1,\ldots,n]$, we take our desired chain $K \in C_{n+1}(\cC)$ to be $K \coloneqq D(I^\ast)$. By property (2), we have that
\begin{align*}
	\partial K = T - S - D(\partial I^{\ast}).
\end{align*}
Moreover, by property (1) we have that $D(J) \in C_{|J|}(\cC[V_J])$ for all oriented simplices $J$ in the support of $\partial I^{\ast}$. This shows that $\partial K$ has colorful support $T - S$.
\end{proof}

To finish our homological proof of Theorem \ref{thm:topological-reconfiguration-graph}, we now combine Lemma \ref{lem:chain-homotopy} with a suitable combinatorial argument. This last step can also be done more algebraically, as in Section \ref{sec:higher-1}.

\begin{proof}[Proof of Theorem \ref{thm:topological-reconfiguration-graph} for $\eta = \eta_H$]
We start by fixing some total ordering $\prec$ on the vertices of $\cC$ with the property that $u \prec v$ whenever $u \in V_i$, $v \in V_j$, and $i < j$. Say that an oriented simplex $[ v_0, \ldots, v_k ]$ is in \textit{elementary form} if it respects this total ordering, that is, $v_0 \prec \cdots \prec v_k$.

Now let $S \coloneqq [ u_0, \ldots, u_n ]$ and $T \coloneqq [ v_0, \ldots, v_n ]$ be any two oriented colorful simplices of $(\cC, \cV)$ such that $u_i, v_i \in V_i$ for all $i$. The goal is to show that (the unoriented counterparts of) $S$ and $T$ lie in the same connected component of the reconfiguration graph $\recongraph(\cC, \cV)$. Our theorem's assumption on $\eta_H$ implies that the hypothesis of Lemma \ref{lem:chain-homotopy} is satisfied with $\cR = \mathbb{Q}$. This yields a chain $K \in C_{n+1}(\cC)$ whose boundary $\partial K$ has colorful support $T - S$. Assume that $K = a_1 K_1 + \cdots + a_N K_N$, where the $a_i$ are coefficients in $\cR - \{0\}$ and the $K_i$ are distinct oriented $(n+1)$-simplices in elementary form. 

We construct a directed graph $G$ with edge weights in $\cR - \{0\}$ as follows. The vertex set of $G$ consists of all oriented colorful simplices $\sigma$ of $(\cC, \cV)$ in elementary form. We put a directed edge $(\sigma_0, \sigma_1)$ in $G$ pointing from $\sigma_0$ to $\sigma_1$ if there exists some oriented $(n+1)$-simplex $K_i$, with $i \in \{1, \ldots, N\}$, whose boundary $\partial K_i$ has colorful support $\sigma_1 - \sigma_0$. In this case, such a $K_i$ is unique and the directed edge $(\sigma_0, \sigma_1)$ is given weight $w(\sigma_0, \sigma_1) = a_i$ corresponding to the coefficient of $K_i$. Notice that the underlying unweighted, undirected graph of $G$ is a subgraph of the reconfiguration graph $\recongraph(\cC, \cV)$. Thus, to prove the theorem it suffices to show that $S$ and $T$ lie in the same weakly connected component of $G$.

Let $\delta^{+}(\sigma)$ denote the set of directed edges of $G$ pointing toward the vertex $\sigma$, and let $\delta^{-}(\sigma)$ denote the set of directed edges pointing away from $\sigma$. The \textit{excess} of vertex $\sigma$ is defined to be the quantity
\begin{align*}
	\sum_{e \in \delta^{+}(\sigma)} w(e) - \sum_{e \in \delta^{-}(\sigma)} w(e).
\end{align*}
The property that $\partial K$ has colorful support $T - S$ is equivalent to saying that the excess of every vertex $\sigma$ of $G$ is $-1$ if $\sigma = S$, is $+1$ if $\sigma = T$, and is $0$ for all other $\sigma$. A well-known fact is that in any weakly connected component $C$ of $G$, the sum of the excesses of all vertices of $C$ is $0$. This follows from the equality
\begin{align*}
	\sum_{\sigma \in C} \sum_{e \in \delta^{+}(\sigma)} w(e) = \sum_{\sigma \in C} \sum_{e \in \delta^{-}(\sigma)} w(e),
\end{align*}
which holds because both sides calculate the total weight of all directed edges that lie between two vertices in $C$. Therefore, $S$ and $T$ lie in the same weakly connected component of $G$.
\end{proof}

\section{Higher dimensional topological Hall theorems} \label{sec:higher}

In this section, we prove Theorem \ref{thm:topological-reconfiguration-complex}, which is a higher dimensional homological generalization of the usual topological Hall theorem (Theorem \ref{thm:topological-hall}) and our reconfiguration variation of it (Theorem \ref{thm:topological-reconfiguration-graph}). In the second subsection, we also describe a variation of Theorem \ref{thm:topological-reconfiguration-complex} coming from homological nerve theorems. Like in Section \ref{sec:proofs}, our vertex partition will be of the form $\cV \coloneqq \{V_0, \ldots, V_n\}$.

\subsection{Proof of Theorem \ref{thm:topological-reconfiguration-complex}} \label{sec:higher-1}

For a complex $\cC$ and partition $\cV = \{V_0, \ldots, V_n\}$ of $V(\cC)$, recall that the \textit{colorful complex} $\colorfulcomplex(\cC, \cV)$ is the abstract simplicial complex with vertex set consisting of all simplices $\sigma$ of $\cC$ that span the classes of $\cV$ (and thus contain a colorful simplex), and a collection $\sigma_0, \ldots, \sigma_k$ of such simplices of $\cC$ form a simplex of $\colorfulcomplex(\cC, \cV)$ whenever $\sigma_0 \subset \cdots \subset \sigma_k$. The colorful complex $\colorfulcomplex(\cC, \cV)$ is the barycentric subdivision of a natural polyhedral complex which we denote $\colorfulcomplex^{\mathrm{P}}(\cC, \cV)$. Specifically, the cells of $\colorfulcomplex^{\mathrm{P}}(\cC, \cV)$ are indexed by all simplices of $\cC$ spanning the classes of $\cV$, and the closure of a cell $T$ indexed by $\tau$ consists of all cells $S$ indexed by a simplex $\sigma$ of $\cC$ spanning the classes of $\cV$ and satisfying $\sigma \subseteq \tau$. See Figure \ref{fig:colorful-complex} for illustration. For a more geometric view, consider a geometric realization $\cK$ of the complex $\cC$ with polyhedron $\lVert \cK \rVert \in \mathbb{R}^d$. In the barycenter of the embedding of each colorful simplex $\sigma$ of $(\cC, \cV)$, we put a vertex $v_{\sigma}$. Then the closure of a cell $T$ of $\colorfulcomplex^{\mathrm{P}}(\cC, \cV)$, indexed by the simplex $\tau$ of $\cC$, is the convex hull of the vertices $v_{\sigma}$ over all colorful simplices $\sigma$ of $(\cC, \cV)$ satisfying $\sigma \subseteq \tau$. Such a cell $T$ is a polyhedron of the form $\Delta_0 \times \cdots \times \Delta_n$ where $\Delta_i$ is an embedded simplex with $|\tau \cap V_i|$ vertices, and thus $T$ has dimension $|\tau| - n - 1$. The 1-skeleton of $\colorfulcomplex^{\mathrm{P}}(\cC, \cV)$ is the reconfiguration graph $\recongraph(\cC, \cV)$.

\begin{figure}
\begin{center}
	\leavevmode
	\includegraphics[scale=0.85]{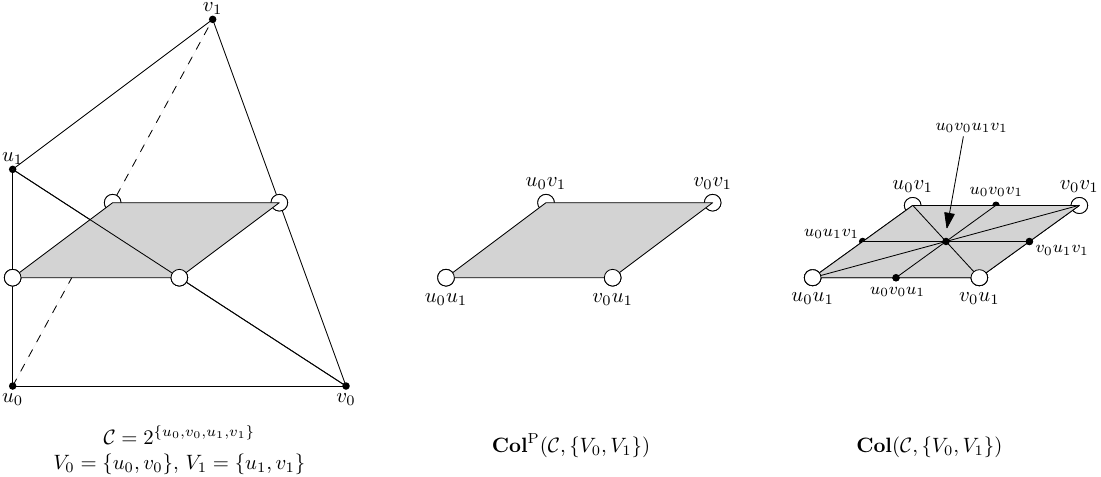}
\end{center}
\caption{An illustration of the colorful complex $\colorfulcomplex(\cC, \cV)$ and the corresponding polyhedral complex $\colorfulcomplex^{\mathrm{P}}(\cC, \cV)$, where $\cC$ is the simplex with vertices $\{u_0, v_0, u_1, v_1\}$, and $\cV = \{V_0, V_1\}$ is given by $V_0 = \{u_0, v_0\}$, $V_1 = \{u_1, v_1\}$.}
\label{fig:colorful-complex}
\end{figure}

To prove Theorem \ref{thm:topological-reconfiguration-complex}, it suffices to prove the following stronger result, where the homology coefficients are over any fixed commutative ring $\cR$.

\begin{thm} \label{thm:high-dimensional-homology}
Let $\cC$ be a complex, let $\cV = \{V_0, \ldots, V_n\}$ be a partition of $V = V(\cC)$, and let $m \ge 0$ be an integer. If
\begin{align*}
	\widetilde{H}_{|I|+m-2}(\cC[V_I]) = 0 \hspace{25pt} \text{for all nonempty } I \subseteq \{0,1,\ldots,n\},
\end{align*}
then $\widetilde{H}_{m-1}(\mathbf{Col}(\cC,\cV)) = 0$.
\end{thm}

The proof idea is the same as the homological proof of Theorem \ref{thm:topological-reconfiguration-graph} that was given in Section \ref{sec:homological-proof}. Namely, we wish to construct a homomorphism $D = D_p : C_{p}(\Delta^n) \rightarrow C_{p+m}(\cC)$ satisfying certain chain homotopy and Sperner-type conditions on the faces of the abstract $n$-dimensional simplex $\Delta^n$, which we again assume has vertex set $\{0,1,\ldots,n\}$. However, before constructing $D$ we need some additional algebraic steps, which make up the bulk of this section. We assume for convenience that $m \ge 1$, but a modified version of this argument works for $m = 0$ (basically the same proof as \cite[Proposition 2.6]{deloera2019discrete}). Like before, we fix a total ordering $\prec$ on the vertices of $\cC$ with the property that $u \prec v$ whenever $u \in V_i$, $v \in V_j$, and $i < j$. We say that an oriented simplex $[ v_0, \ldots, v_p ] \in C_p(\cC)$ is in \textit{elementary form} if $v_0 \prec \cdots \prec v_p$. Likewise, we say that oriented simplex $I = [i_0, \ldots, i_p] \in C_p(\Delta^n)$ is in \textit{elementary form} if $i_0 < \cdots < i_p$.

First, we define linear functions that functorially relate the original complex $\cC$ to the colorful complex $\colorfulcomplex(\cC, \cV)$. These are modifications of the well-known algebraic subdivision operator $\mathrm{sd}$ (see \cite[Chapter 2.17]{munkres2018elements}). For a simplex (or oriented simplex) $\sigma$ of $\cC$ that spans the classes of $\cV$, we denote by $\hat{\sigma}$ the corresponding vertex of $\colorfulcomplex(\cC, \cV)$. Fixing an oriented simplex $I$ on $\Delta^n$, we define a linear function
\begin{align*}
	\colorfulsubdivision^I = (\colorfulsubdivision^I)_p : C_{p+|I|-1}(\cC[V_I]) \rightarrow C_{p}(\colorfulcomplex(\cC[V_I], \{V_i : i \in I\}))
\end{align*}
inductively on $p \ge 0$ as follows. For $p = 0$ and an oriented simplex $\sigma \in C_{|I|-1}(\cC[V_I])$ in elementary form, we put $\colorfulsubdivision^I(\sigma) \coloneqq [\hat{\sigma}]$ if $\sigma$ is a colorful simplex of $(\cC[V_I], \{V_i : i \in I\})$, and we put $\colorfulsubdivision^I(\sigma)=0$ otherwise. We extend $\colorfulsubdivision^I$ to all chains in $C_{|I|-1}(\cC[V_I])$ by linearity. For $p \ge 1$ and an oriented simplex $\sigma \in C_{p+|I|-1}(\cC[V_I])$ in elementary form, we put
\begin{align*}
	\colorfulsubdivision^I(\sigma) \coloneqq [\hat{\sigma}, \colorfulsubdivision^I(\partial \sigma)].
\end{align*}
(For a vertex $w$ and chain $c \coloneqq a_1 \sigma_1 + \cdots + a_k \sigma_k$, where $a_i \in \cR$ and $\sigma_i$ are oriented simplices, we use the bracket notation $[w, c] \coloneqq a_1[w, \sigma_1] + \cdots + a_k[w, \sigma_k]$.) Again we extend to all chains in $C_{p+|I|-1}(\cC[V_I])$ by linearity. At a high level, we are associating every oriented simplex $\sigma$ with its corresponding oriented cell in $\mathbf{Col}^{\mathrm{P}}(\cC[V_I], \{V_i : i \in I\})$, and constructing the simplicial chain on the subdivided cell that is consistent with the orientation of $\sigma$. Notice that $f^I(\sigma) = 0$ whenever $\sigma$ is an oriented simplex on $\cC[V_I]$ that does not span the color classes $\{V_i : i \in I\}$. We start with the following lemma.

\begin{lem} \label{lem:boundary-commute}
The map $\colorfulsubdivision^I$ commutes with the boundary operator $\partial$.
\end{lem}

\begin{proof}
By linearity, it suffices to show that $\partial \colorfulsubdivision^I(\sigma) = \colorfulsubdivision^I(\partial \sigma)$ for every oriented simplex $\sigma \in C_{p+|I|-1}(\cC)$ in elementary form. This is straightforward for $p=0$. For $p \ge 1$, by induction we have
\begin{align*}
	\partial \colorfulsubdivision^I(\sigma) = \colorfulsubdivision^{I}(\partial \sigma) - [\hat{\sigma}, \partial (\colorfulsubdivision^I(\partial \sigma))] = \colorfulsubdivision^{I}(\partial \sigma) - [\hat{\sigma}, \colorfulsubdivision^I(\partial^2 \sigma)] = \colorfulsubdivision^{I}(\partial \sigma),
\end{align*}
as required.
\end{proof}

For convenience, we now put
\begin{align*}
	\colorfulsubdivision \coloneqq \colorfulsubdivision^{[0,1,\ldots,n]}.
\end{align*}
Let $C_{p}^{\mathrm{col}}(\cC)$ denote the subgroup of the chain group $C_{p}(\cC)$ supporting only those oriented simplices that span the classes of $\cV$. The following lemma allows us to focus our attention in Theorem \ref{thm:high-dimensional-homology} to cycles that take a particular form. Two chains are called \textit{homologous} if their difference is a boundary.

\begin{lem} \label{lem:universal-cycle}
Every $k$-cycle of $\colorfulcomplex(\cC, \cV)$ is homologous to a cycle of the form $\colorfulsubdivision(R)$ for some chain $R \in C_{k+n}^{\mathrm{col}}(\cC)$.
\end{lem}

\begin{proof}
This is a standard cellular approximation argument. Recall that the simplicial complex $\colorfulcomplex(\cC, \cV)$ is the barycentric subdivision of the polyehdral complex $\colorfulcomplex^{\mathrm{P}}(\cC, \cV)$. Notice also that $\colorfulcomplex(\cC^{(k+n)}, \cV)$ is the barycentric subdivision of the polyehdral $k$-skeleton $(\colorfulcomplex^{\mathrm{P}}(\cC, \cV))^{(k)}$. By the exactness axiom of reduced homology, the following sequence is exact:
\begin{align*}
	\widetilde{H}_{k}(\colorfulcomplex(\cC^{(k+n)}, \cV)) \xlongrightarrow{i_{\ast}} \widetilde{H}_{k}(\colorfulcomplex(\cC, \cV)) \xlongrightarrow{j_{\ast}} H_{k}(\colorfulcomplex(\cC, \cV), \colorfulcomplex(\cC^{(k+n)},\cV)),
\end{align*}
where $i_{\ast}$ and $j_{\ast}$ are induced from the inclusion maps $i : \colorfulcomplex(\cC^{(k+n)}, \cV) \rightarrow \colorfulcomplex(\cC, \cV)$ and $j : (\colorfulcomplex(\cC, \cV), \emptyset) \rightarrow (\colorfulcomplex(\cC, \cV), \colorfulcomplex(\cC^{(k+n)},\cV))$, respectively. By basic properties of cellular homology, 
\begin{align*}
	H_{k}(\colorfulcomplex(\cC, \cV), \colorfulcomplex(\cC^{(k+n)},\cV)) \cong H_{k}^{\mathrm{cell}}(\colorfulcomplex^{\mathrm{P}}(\cC, \cV), (\colorfulcomplex^{\mathrm{P}}(\cC, \cV))^{(k)}) = 0.
\end{align*}
This implies that the map $i_{\ast}$ is surjective, and hence every $k$-cycle $\colorfulcomplex(\cC, \cV)$ is homologous to a cycle in $Z_{k}(\colorfulcomplex(\cC^{(k+n)}, \cV))$. 
	
Now it suffices to show that every cycle $c \in Z_{k}(\colorfulcomplex(\cC^{(k+n)}, \cV))$ is of the form $f(R)$ for some chain $R \in C_{k+n}^{\mathrm{col}}(\cC)$. By construction, we can write $c$ uniquely as $c = \sum_{\sigma} [\hat{\sigma}, d_{\sigma}]$, where the sum is over all oriented simplices $\sigma \in C_{k}^{\mathrm{col}}(\cC)$ in elementary form, and $d_{\sigma}$ is some chain in $C_{k-1}(\colorfulcomplex(\cC^{(k+n)}, \cV))$ whose support lies in the link of $\hat{\sigma}$ in $\colorfulcomplex(\cC^{(k+n)}, \cV)$. Note that
\begin{align*}
	0 = \partial c = \sum_{\sigma} \partial [\hat{\sigma}, d_{\sigma}] = \sum_{\sigma} (d_{\sigma} - [\hat{\sigma}, \partial d_{\sigma}]).
\end{align*}
This implies that $\partial d_{\sigma} = 0$ for all $\sigma$. Since the link of $\hat{\sigma}$ in $\colorfulcomplex(\cC^{(k+n)}, \cV)$ is a triangulated $(k-1)$-sphere, up to constant factors $d_{\sigma}$ is the boundary of the star of $\hat{\sigma}$, i.e., $d_{\sigma} = a_{\sigma} \partial f(\sigma)$ for some $a_{\sigma} \in \cR$. Therefore,
\begin{align*}
	c = \sum_{\sigma} [\hat{\sigma}, a_{\sigma} \partial f(\sigma)] = \sum_{\sigma} a_{\sigma}[\hat{\sigma}, f(\partial \sigma)] = f\left(\sum_{\sigma} a_{\sigma} \sigma\right),
\end{align*}
as required.
\end{proof}

For our next step, we define a family of bilinear functions
\begin{align*}
	\bilinearfunction = \bilinearfunction_{p,q} : C_{p}^{\text{col}}(\cC) \times C_{q}(\Delta^n) \rightarrow C_{p+q-n}(\cC),
\end{align*}
which will indicate how we restrict a chain on $\cC$ to a lower-dimensional chain whose vertices have colors in a given subset $I \subseteq \{0,1,\ldots,n\}$. Fix an oriented simplex $I \in C_{q}(\Delta^n)$ in elementary form. Consider an oriented simplex $\sigma \in C_{p}^{\text{col}}(\cC)$ of the form
\begin{align*}
	\sigma \coloneqq [\sigma_0, \sigma_1],
\end{align*}
where $\sigma_0$ is an oriented simplex of $\cC[V_I]$, $\sigma_1$ is an oriented simplex of $\cC[V - V_I]$, and $\sigma_0$ and $\sigma_1$ are each in elementary form. Call such a representation of oriented simplex $\sigma$ as being \textit{$I$-elementary}. We put
\begin{align*}
	\bilinearfunction_{p,q}(\sigma, I) \coloneqq \sigma_0
\end{align*}
if $|\sigma_1| = n-q$ (and $|\sigma_0| = p+q-n+1$), and we put $\bilinearfunction_{p,q}(\sigma, I) = 0$ otherwise. In other words, we either prune the last $n-q$ coordinates of $\sigma$ (removing the vertices not in $V_I$), or we send to $0$. Since $\sigma$ contains a colorful simplex, we always have $|\sigma_1| \ge n - q$. We extend $\bilinearfunction_{p,q}$ to all of $C_{p}(\cC) \times C_{q}(\Delta^n)$ by bilinearity. The following lemma partly explains the motivation for this definition.

\begin{lem} \label{lem:basic-boundary}
For an oriented simplex $I \in C_q(\Delta^n)$ and color $j \notin I$, the chain $\bilinearfunction_{p,q}(R, I)$ is the support of $\partial \bilinearfunction_{p,q+1}(R, [j, I])$ on the oriented simplices whose vertices lie entirely in $V_I$.
\end{lem}

\begin{proof}
By the linearity of $\bilinearfunction(\cdot, I)$, it suffices to show this when the chain $R$ is an oriented simplex $\sigma \in C_p^{\mathrm{col}}(\cC)$. Assume that $I$ is in elementary form, and $\sigma \coloneqq [\sigma_0, \sigma_1]$ is in $I$-elementary form. Also call $I'$ the result of putting $[j, I]$ in elementary form, so that $I' = (-1)^{|\{i \in I : i < j\}|} [j, I]$. Then say that $\sigma = \pm [\sigma_0', \sigma_1']$, where $[\sigma_0', \sigma_1']$ is in $I'$-elementary form. There are three cases:
\begin{itemize}
	\item[(i)] $|\sigma_1'| > n - q - 1$. Then we also have $|\sigma_1| > n - q$, and hence $\bilinearfunction_{p,q}(\sigma, I) = 0 = \partial \bilinearfunction_{p,q+1}(\sigma, [j, I])$.
	
	\item[(ii)] $|\sigma_1'| = n - q - 1$ and $|\sigma_1| > n - q$. In this case, $\sigma_0'$ contains more than one vertex of color $j$, and thus the support of $\partial \bilinearfunction_{p,q+1}(\sigma, [j, I]) = \pm \partial \sigma_0'$ on the oriented simplices whose vertices lie entirely in $V_I$ is $0$. In addition, we have $\bilinearfunction_{p,q}(\sigma, I) = 0$ because $|\sigma_1| > n - q$.
		
	\item[(iii)] $|\sigma_1'| = n - q - 1$ and $|\sigma_1| = n - q$. Then $\sigma_0'$ has a unique vertex of color $j$, and we can write 
	\begin{align*}
		\partial g_{p, q+1}(\sigma, [j, I]) = (-1)^{|\{i \in I : i < j\}|} \partial g_{p, q+1}(\sigma, I') = (-1)^{|\{i \in I : i < j\}|} \partial \sigma_0'.
	\end{align*}
	On the right-hand side, it is easy to see that the unique oriented simplex term whose vertices lie entirely in $V_I$ is $\sigma_0$. And we also have that $g_{p,q}(\sigma, I) = \sigma_0$.
\end{itemize}
This finishes the proof.
\end{proof}

The next lemma states that the function $g$ commutes with the boundary operator $\partial$ in the first input.

\begin{lem} \label{lem:commute-first}
We have $\partial \bilinearfunction(R, I) = \bilinearfunction(\partial R, I)$.
\end{lem}

\begin{proof}
By the linearity of $\bilinearfunction(\cdot, I)$, it suffices to verify that $\partial \bilinearfunction_{p,q}(\sigma, I) = \bilinearfunction_{p-1,q}(\partial \sigma, I)$ for every oriented simplex $\sigma \in C_{q}^{\text{col}}(\cC)$ in $I$-elementary form, say $\sigma \coloneqq [\sigma_0, \sigma_1]$. We have
\begin{align*}
	\bilinearfunction_{p-1,q}( \partial \sigma, I) = \bilinearfunction_{p-1,q}([\partial \sigma_0, \sigma_1], I) + (-1)^{|\sigma_0|-1} \bilinearfunction_{p-1,q}([\sigma_0, \partial \sigma_1], I).
\end{align*}
There are three cases:
\begin{itemize}
	\item[(i)] $|\sigma_1| > n - q + 1$. Then we already have $\partial \bilinearfunction_{p,q}(\sigma, I) = 0 = \bilinearfunction_{p-1,q}(\partial \sigma, I)$.
	
	\item[(ii)] $|\sigma_1| = n - q + 1$. Then the first term of the right-hand side is $0$. The second term of the right-hand side is also $0$ because $\partial \sigma_1$ supports two oriented colorful simplices of $(\cC[V - V_I], \{V_i : i \notin I\})$ with opposite signs, and thus $\bilinearfunction_{p-1,q}([\sigma_0, \partial \sigma_1]) = \sigma_0 - \sigma_0 = 0$. Hence, $\partial \bilinearfunction_{p,q}(\sigma, I) = 0 = \bilinearfunction_{p-1,q}(\partial \sigma, I)$.
	
	\item[(iii)] $|\sigma_1| = n - q$. Then the first term of the right-hand side is $\partial \sigma_0$ and the second term is $0$. Therefore, $\partial \bilinearfunction_{p,q}(\sigma, I) = \partial \sigma_0 = \bilinearfunction_{p-1,q}(\partial \sigma, I)$.
\end{itemize}
This finishes the proof.
\end{proof}

Our proof of Theorem \ref{thm:high-dimensional-homology} will require the additional property that $g$ commutes with the boundary operator in the second input: $\partial \bilinearfunction(R, I) = \bilinearfunction(R, \partial I)$. This is not true for general chains $R$, but our goal is to show that it is true for chains $R$ satisfying $\partial \colorfulsubdivision(R) = 0$, which is all that we need. For showing this, we let
\begin{align*}
	\pi^I : \colorfulcomplex(\cC, \cV) \rightarrow \colorfulcomplex(\cC[V_I], \{V_i : i \in I\})
\end{align*}
denote the simplicial map that sends a vertex $\hat{\sigma}$ of $\colorfulcomplex(\cC, \cV)$, corresponding to simplex $\sigma$ of $\cC$, to the vertex $\hat{\tau}$ of $\colorfulcomplex(\cC[V_I], \{V_i : i \in I\})$, corresponding to simplex $\tau \coloneqq \sigma \cap V_I$ of $\cC[V_I]$. Since $\pi^I$ is a simplicial map, it induces a chain map 
\begin{align*}
	\pi_{\#}^I : C_p(\colorfulcomplex(\cC, \cV)) \rightarrow C_p(\colorfulcomplex(\cC[V_I], \{V_i : i \in I\})).
\end{align*}
The following lemma explains how $\pi_{\#}^I$ relates to the previously defined maps.

\begin{lem} \label{lem:pi-lemma}
We have $\pi_{\#}^I(\colorfulsubdivision(R)) = \colorfulsubdivision^I(\bilinearfunction(R, I))$.
\end{lem}

\begin{proof}
By linearity, it suffices to verify this when $R$ is an oriented simplex $\sigma \in C_p^{\mathrm{col}}(\cC)$ in $I$-elementary form, say $\sigma \coloneqq [\sigma_0, \sigma_1]$. We apply induction on $|\sigma|$. If $|\sigma| = n+1$, meaning that $\sigma$ is an oriented colorful simplex of $(\cC, \cV)$, then
\begin{align*}
	\pi_{\#}^I(\colorfulsubdivision(\sigma)) = \pi_{\#}^I([\hat{\sigma}]) = [\hat{\sigma}_0] = \colorfulsubdivision^I(\sigma_0) = \colorfulsubdivision^I(\bilinearfunction(\sigma, I)).
\end{align*}
If $|\sigma| > n+1$, then by induction we have 
\begin{align*}
	\pi_{\#}^I(\colorfulsubdivision(\sigma)) &= \pi_{\#}^I([\hat{\sigma}, \colorfulsubdivision(\partial \sigma)]) = [\pi^I(\hat{\sigma}), \pi_{\#}^I(\colorfulsubdivision(\partial \sigma))] \\
	&= [\hat{\sigma}_0, \colorfulsubdivision^{I}(\bilinearfunction(\partial \sigma, I))] = [\hat{\sigma}_0, \colorfulsubdivision^{I}(\partial g(\sigma, I))].
\end{align*}
If $g(\sigma, I) = 0$, then we also have that $\pi_{\#}^I(\colorfulsubdivision(\sigma)) = 0$. Otherwise, if $g(\sigma, I) = \sigma_0$, then we get that $\pi_{\#}^I(\colorfulsubdivision(\sigma)) = [\hat{\sigma}_0, \colorfulsubdivision^{I}(\partial \sigma_0)] = \colorfulsubdivision^I(g(\sigma, I))$. This finishes the proof.
\end{proof}

Now we prove the desired property that $g$ commutes with the boundary operator $\partial$ in the second input.

\begin{lem} \label{lem:commute-second}
If chain $R \in C_{m+n-1}^{\mathrm{col}}(\cC)$ satisfies $\partial \colorfulsubdivision(R) = 0$, then $\partial \bilinearfunction(R, I) = \bilinearfunction(R, \partial I)$.
\end{lem}

\begin{proof}
Fix an oriented simplex $I \in C_q(\Delta^n)$ in elementary form, say $I \coloneqq [i_0, \ldots, i_q]$, and put $I_k \coloneqq [i_0, \ldots, \hat{i}_k, \ldots, i_q]$. By Lemma \ref{lem:basic-boundary}, $g_{p,q-1}(R, I_k)$ is precisely the support of the chain $\partial g_{p,q}(R, [i_k, I_k]) = (-1)^k \partial g_{p,q}(R, I)$ on oriented simplices consisting only of vertices from $V_{I_k}$. Assume for now that every oriented simplex in the support of $\partial \bilinearfunction_{p,q}(R, I)$ does not span the classes $\{V_i : i \in I\}$, i.e., it misses some color $i_k$ from $I$. Then we would get that
\begin{align*}
	\partial g_{p,q}(R, I) = \sum_{k=0}^q (-1)^k g_{p,q-1}(R, I_k) = g_{p,q-1}(R, \partial I),
\end{align*}
which would finish the proof of the lemma.

What is left to show is that every oriented simplex in the support of $\partial \bilinearfunction_{p,q}(R, I)$ misses some color $i_k$ from $I$. By Lemma \ref{lem:pi-lemma}, Lemma \ref{lem:boundary-commute}, and the assumption that $\partial \colorfulsubdivision(R) = 0$, we have 
\begin{align*}
	\colorfulsubdivision^I(\partial \bilinearfunction(R, I)) = \colorfulsubdivision^I(\bilinearfunction(\partial R, I)) = \pi_{\#}^I(\colorfulsubdivision(\partial R)) = \pi_{\#}^I(\partial \colorfulsubdivision(R)) = 0.
\end{align*}
This shows that $\partial \bilinearfunction(R, I)$ lies in the kernel of $\colorfulsubdivision^I$. On the other hand, notice that $\colorfulsubdivision^I$ is injective on oriented simplices of $\cC[V_I]$ that span the classes $\{V_i : i \in I\}$. This means that all the oriented simplices in the support of $\partial \bilinearfunction(R, I)$ do not span the classes $\{V_i : i \in I\}$, i.e., they miss some color from $I$, as required.
\end{proof}

The proof of Theorem \ref{thm:high-dimensional-homology} now proceeds basically the same as our proof of Lemma \ref{lem:chain-homotopy} in Section \ref{sec:homological-proof}. We only sketch this part. The goal is to show that every $(m-1)$-cycle of $\colorfulcomplex(\cC, \cV)$ is an $(m-1)$-boundary. By Lemma \ref{lem:universal-cycle}, it suffices to show this for cycles of the form $\colorfulsubdivision(R)$ for some fixed chain $R \in C_{m+n-1}^{\mathrm{col}}(\cC)$. Since $\partial \colorfulsubdivision(R) = 0$, Lemma \ref{lem:commute-second} applies. We follow the same inductive procedure as the proof of Lemma \ref{lem:chain-homotopy} to construct a homomorphism $D = D_p : C_{p}(\Delta^n) \rightarrow C_{p+m}(\cC)$ that satisfies the following two properties for all oriented simplices $I \in C_p(\Delta^n)$:
\begin{itemize}
	\item[(1)] $D(I) \in C_{|I|+m-1}(\cC[V_I])$, 
	\item[(2)] $\partial D(I) + D(\partial I) = g(R, I)$.
\end{itemize}
The construction of such a homomorphism $D$ is where Lemma \ref{lem:commute-second} is needed. Now taking $I^{\ast} \coloneqq [0,1,\ldots,n]$, we get that $\partial D(I^\ast) + D(\partial I^\ast) = R$ by property (2). Because $\colorfulsubdivision(D(\partial I^\ast)) = 0$ by property (1), we have by Lemma \ref{lem:boundary-commute} that
\begin{align*}
	\partial \colorfulsubdivision(D(I^\ast)) = \colorfulsubdivision( \partial D(I^\ast)) = \colorfulsubdivision(\partial D(I^\ast) + D(\partial I^\ast)) = f(R).
\end{align*}
This demonstrates that $f(R)$ is indeed a boundary, concluding the proof of Theorem \ref{thm:high-dimensional-homology}.

\subsection{A nerve variation}

In this subsection, we prove a similar result to Theorem \ref{thm:topological-reconfiguration-complex} that comes from homological nerve theorems. This nerve approach more closely follows Meshulam's original proof \cite{meshulam2003domination, meshulam2001clique} of the topological Hall theorem, but the conclusion is weaker in the case of reconfigurations $m=1$.

Given a complex $\cC$ and a finite collection $\cF$ of subcomplexes of $\cC$, the \textit{nerve} $\nerve(\cF)$ of $\cF$ is the complex on vertex set $\cF$ whose simplices are the subsets of $\cF$ with nonempty intersection. In our current setting, let $\cC$ be a complex and let $\cV = \{V_0, \ldots, V_n\}$ be a partition of $V(\cC)$. Let $\cF$ be the collection of maximal simplices of $\cC$ that span the classes of $\cV$, and define the \textit{colorful nerve} to be the complex $\colorfulnerve(\cC,\cV) \coloneqq \nerve(\cF)$. We quickly show the following variation of Theorem \ref{thm:topological-reconfiguration-complex}.

\begin{thm} \label{thm:colorful-simplex-nerve}
Let $\cC$ be a complex, let $\cV = \{V_1, \ldots, V_n\}$ be a partition of $V(\cC)$, and let $m \ge 0$ be an integer. If 
\begin{align*}
	\eta_H(\cC[V_I]) \ge |I| + m \hspace{30pt} \text{for all nonempty } I \subseteq [n],
\end{align*}
then $\eta_H(\colorfulnerve(\cC, \cV)) \ge m+1$.
\end{thm}

In the case of reconfigurations $m = 1$, the conclusion of Theorem \ref{thm:colorful-simplex-nerve} is that for any two colorful simplices $S, T$ of $(\cC, \cV)$, there is a sequence $S = S_0, S_1, \ldots, S_N = T$ of colorful simplices of $(\cC, \cV)$ such that $S_{j - 1} \cap S_j \neq \emptyset$ for all $1 \le j \le N$. This is weaker than the conclusion of Theorem \ref{thm:topological-reconfiguration-graph}. In the language of hypergraphs, Theorem \ref{thm:colorful-simplex-nerve} only allows one to conclude the existence of \textit{loose walks} between colorful simplices, whereas Theorem \ref{thm:topological-reconfiguration-graph} allows one to conclude the existence of \textit{tight walks} between colorful simplices.

Proceeding to the proof of Theorem \ref{thm:colorful-simplex-nerve}, we apply the following homological nerve theorem of Montejano \cite{montejano2017variation} (see also \cite{meunier2020different}), which is a refinement of Meshulam's original nerve theorem \cite{meshulam2001clique}. It applies with coefficients over any fixed field (an analogous statement also holds for coefficients in any finitely generated abelian group).

\begin{thm}[\cite{montejano2017variation}] \label{lem:nerve}
Let $\cC$ be a complex, let $\cF$ be a finite collection of subcomplexes of $\cC$ such that $\bigcup \cF = \cC$, and let $-1 \le k < |\cF|$ be an integer. If
\begin{align*}
	\widetilde{H}_{k - |\sigma|}\bigl( \textstyle{\bigcap \sigma} \bigr) = 0 \hspace{30pt} \text{for all } \sigma \in \nerve(\cF) \text{ with } 1 \le |\sigma| \le k+1,
\end{align*}
then $\dim \widetilde{H}_{k}(\nerve(\cF)) \le \dim \widetilde{H}_{k}(\cC)$.
\end{thm}

We now use Theorem \ref{lem:nerve} to prove the following stronger version of Theorem \ref{thm:colorful-simplex-nerve}, again with coefficients in a fixed field.

\begin{thm}
Let $\cC$ be a complex, let $\cV = \{V_0, \ldots, V_n\}$ be a partition of $V(\cC)$, and let $m \ge 0$ be an integer. If
\begin{align*}
	\widetilde{H}_{|I| + m - 2}\left( \cC[V_I] \right) = 0 \hspace{30pt} \text{for all nonempty } I \subseteq \{0,1,\ldots,n\},
\end{align*}
then $\widetilde{H}_{m - 1}(\colorfulnerve(\cC, \cV)) = 0$.
\end{thm}

\begin{proof}
For $V \coloneqq V(\cC)$, we define the subcomplex $\cD_j \coloneqq \cC[V - V_j]$ for all $j \in \{0,1,\ldots,n\}$, and we put $\cF' \coloneqq \{\cD_0, \ldots, \cD_n\}$. We apply Theorem \ref{lem:nerve} to the collection $\cF \coloneqq \cF' \cup V(\colorfulnerve(\cC, \cV))$ and the integer $k = n + m - 1$. It is easy to see that $\bigcup \cF = \cC$. Now consider any $\sigma \in \nerve(\cF)$ with $|\sigma| \ge 1$. Then we can write $\sigma = \{\cD_j : j \in J\} \cup \{ \cE_i : i \in \{1,\ldots,\ell\} \}$, where $J$ is some proper subset of $\{0,1,\ldots,n\}$ and $\cE_i \in V(\colorfulnerve(\cC, \cV))$ for all $i \in \{1,\ldots,\ell\}$. Denote $I = I(J) \coloneqq \{0,1,\ldots,n\} - J$, so that $I$ is nonempty. 
	
We verify that $\cF$ satisfies the conditions of Theorem \ref{lem:nerve}. First assume that $\ell = 0$, i.e., $\sigma = \{\cD_j : j \in J\}$. Then we have $\textstyle{\bigcap \sigma} = \cC[V_I]$, and the theorem assumption yields that $\widetilde{H}_{k - |\sigma|}\bigl( \textstyle{\bigcap \sigma} \bigr) = \widetilde{H}_{|I| + m - 2}\left( \cC[V_I] \right) = 0$. Now assume that $\ell \ge 1$. Then $\tau \coloneqq \bigcap_{i=1}^{\ell} \cE_i$ is a nonempty simplex in $\cC$, and hence so is $\bigcap \sigma = \tau \cap V_I$. Thus, we automatically have $\widetilde{H}_{k - |\sigma|}\bigl( \textstyle{\bigcap \sigma} \bigr) = 0$. Thus, the conditions of Theorem \ref{lem:nerve} are satisfied. From this, we deduce that $\dim \widetilde{H}_{n+m-1}(\nerve(\cF)) \le \dim \widetilde{H}_{n+m-1}(\cC)$. The right-hand side is $0$ by the theorem assumption with $I \coloneqq \{0,1,\ldots,n\}$, so we get that $\widetilde{H}_{n+m-1}(\nerve(\cF)) = 0$. 
	
Now observe that the nerve $\nerve(\cF)$ is equal to the join $\nerve(\cF') \ast \colorfulnerve(\cC, \cV)$: It is easy to see that $\nerve(\cF)$ is necessarily contained in $\nerve(\cF') \ast \colorfulnerve(\cC, \cV)$; and, on the other hand, we already argued above that any $\sigma$ in the right-hand side necessarily has nonempty intersection and hence is contained in $\nerve(\cF)$. Also observe that $\nerve(\cF')$ is the boundary of the $n$-dimensional simplex on vertex set $\cF'$, and hence its only non-vanishing reduced homology group is $\widetilde{H}_{n-1}(\nerve(\cF')) \cong \mathbb{F}$. Therefore, applying the K\"unneth formula for joins (see \cite[equation 9.12]{bjorner1995topological}), we have
\begin{align*}
	0 = \widetilde{H}_{n+m-1}(\nerve(\cF)) \cong \bigoplus_{\substack{p+q = \\ n+m-2}} \left[ \widetilde{H}_{p}(\nerve(\cF')) \otimes \widetilde{H}_{q}(\colorfulnerve(\cC, \cV)) \right] \cong \widetilde{H}_{m-1}(\colorfulnerve(\cC, \cV)),
\end{align*}
as we wanted to show.
\end{proof}

\section{Intersections of complexes and matroids} \label{sec:intersections}

In this section, we provide a generalization our topological Hall theorems to matroidal settings, following the seminal work of Aharoni and Berger \cite{aharoni2006intersection}. These results are relevant for the discrete geometry applications in Section \ref{sec:geometric-applications}. Our proof exhibits an interesting application of interval subdivisions of posets.

Going beyond colorful simplices as in the topological Hall theorem (Theorem \ref{thm:topological-hall}), Aharoni and Berger \cite{aharoni2006intersection} proved the following result providing a sufficient condition for the existence of a set of specified size lying in the intersection of a given complex and matroid on the same ground set. It generalizes the hard direction of Edmonds' matroid intersection theorem \cite{edmonds2003submodular}, which is about the intersection of two matroids.

\begin{thm}[\cite{aharoni2006intersection}] \label{thm:aharoni-berger}
Let $\cC$ be a complex and let $\cM$ be a matroid on the same ground set $V$, let $1 \le k \le r(\cM)$ be an integer, and let $\eta \in \{\eta_{\pi}, \eta_H\}$. If
\begin{align*}
	\eta(\cC[X]) + r(\cM[V - X]) \ge k
\end{align*}
for all $X \subseteq V$ for which $V - X$ is a flat of $\cM$, then there exists a set of size $k$ in both $\cC$ and $\cM$.
\end{thm}

\subsection{Reconfiguration results}

First we describe a reconfiguration version of Theorem \ref{thm:aharoni-berger}, thus providing an extension of our reconfiguration topological Hall theorem (Theorem \ref{thm:topological-reconfiguration-graph}) to a matroidal setting. The proof will be in the next subsection, as a special case of our higher dimensional connectedness result (Theorem \ref{thm:complex-matroid-connectedness}).

Let $\cC$ be a complex and let $\cM$ be a matroid on the same ground set $V$. We define the reconfiguration graph $\recongraph(\cC, \cM)$ to have vertex set consisting of all faces of $\cC$ that are also bases of $\cM$, and two such sets $S, T$ form an edge of the graph if $S \cup T$ is a simplex of $\cC$ with cardinality $r(\cM)+1$. Also, given an integer $1 \le k \le r(\cM)$, we denote $\recongraph(\cC, \cM; k) \coloneqq \recongraph(\cC, \cM_{\le k})$, where $\cM_{\le k} \coloneqq \{A \in \cM : |A| \le k\}$ denotes the $k$-truncation of $\cM$. The following is our reconfiguration version of Theorem \ref{thm:aharoni-berger} of Aharoni and Berger \cite{aharoni2006intersection}. 

\begin{thm} \label{thm:complex-matroid-reconfiguration}
Let $\cC$ be a complex and let $\cM$ be a matroid on the same ground set $V$, let $1 \le k \le r(\cM)$ be an integer, and let $\eta \in \{\eta_{\pi}, \eta_H\}$. If
\begin{align*}
	\eta(\cC[X]) + r(\cM[V - X]) \ge k + 1
\end{align*}
for all $X \subseteq V$ for which $V - X$ is a flat of $\cM$ of rank at most $k - 1$, then $\recongraph(\cC, \cM; k)$ is connected. 
\end{thm}

In the case $k = r(\cM)$, one could also state the condition of Theorem \ref{thm:complex-matroid-reconfiguration} in the form
\begin{align*}
	\eta(\cC[X]) \ge r(\cM/(V - X)) + 1
\end{align*}
for all nonempty $X \subseteq V$  for which $V - X$ is a flat of $\cM$. Our reconfiguration version of the topological Hall theorem (Theorem \ref{thm:topological-reconfiguration-graph}) is the special case of Theorem \ref{thm:complex-matroid-reconfiguration} when $k = r(\cM)$ and $\cM$ is a partition matroid. 

We distinguish the special case of the intersection of two matroids $\cM$ and $\cN$ on the same ground set $V$. In this setting, the reconfiguration graph $\recongraph(\cM, \cN; k)$ has vertex set consisting of all common independent sets of $\cM$ and $\cN$ of cardinality $k$, and two such sets $S, T$ form an edge of the graph if $S \cup T$ is an independent set of $\cM$ of cardinality $k+1$. Theorem \ref{thm:complex-matroid-reconfiguration} gives the following.

\begin{thm} \label{thm:matroid-intersection}
Let $\cM$ and $\cN$ be matroids on the same ground set $V$, and let $k \ge 1$ be an integer. If
\begin{align*}
	r(\cM[X]) + r(\cN[V - X]) \ge k+1
\end{align*}
for all $X \subseteq V$ for which $V - X$ is a flat of $\cN$ of rank at most $k - 1$, then $\recongraph(\cM, \cN; k)$ is connected. 
\end{thm}

The matroid intersection theorem of Edmonds \cite{edmonds2003submodular} states that the size $\nu(\cM, \cN)$ of a largest set in common independent set of matroids $\cM$ and $\cN$ on the same ground set $V$ is precisely given by 
\begin{align*}
	\nu(\cM, \cN) = \min \{r(\cM[X]) + r(\cN[V - X]) : X \subseteq V\},
\end{align*}
with the minimum attained by some $X \subseteq V$ for which $V - X$ is a flat of $\cN$. Thus, Theorem \ref{thm:matroid-intersection} implies the following corollary.

\begin{cor} \label{cor:matroid-intersection}
	Let $\cM$ and $\cN$ be matroids on the same ground set $V$. If $k \le \nu(\cM, \cN) - 1$, then $\recongraph(\cM, \cN; k)$ is connected. 
\end{cor}

Corollary \ref{cor:matroid-intersection} is in fact a special case of \cite[Theorem 15]{bousquet2023feedback} by Bousquet, Hommelsheim, Kobayashi, M\"uhlenthaler, and Suzuki on reconfigurations of matroid parity sets. In addition to the intersection of two matroids, matroid parity sets include other structures of interest such as feedback vertex sets of graphs, matchings in general graphs, and connected vertex covers. In the setting of the intersection of two matroids $\cM$ and $\cN$, the reconfiguration graph of \cite{bousquet2023feedback} does not require that adjacent common independent sets $S, T$ satisfy that $S \cup T$ be an independent set of $\cM$, unlike our setting. Nevertheless, their proof method still works under our stronger adjacency condition. On the other hand, their proof method does not immediately extend to Theorem \ref{thm:matroid-intersection} in the case when $k = r(\cN)$, which is allowed because the stated condition is about \textit{nonempty} $X \subseteq V$. It is still possible to prove Theorem \ref{thm:matroid-intersection} using a standard augmenting path argument, but we do not include this here. 

Aharoni and Berger \cite{aharoni2006intersection} used their Theorem \ref{thm:aharoni-berger} to obtain various non-trivial results about the existence, covering, and packing of sets lying in the intersection of any given number of matroids. One may use Theorem \ref{thm:complex-matroid-reconfiguration} to obtain reconfiguration versions of their results.

\subsection{Higher dimensional connectedness}

Now we state and also prove a generalized higher dimensional version of Theorems \ref{thm:aharoni-berger} and \ref{thm:complex-matroid-reconfiguration}, thus giving our matroidal generalization of Theorem \ref{thm:topological-reconfiguration-complex}. 

Let $\cC$ be a complex and let $\cM$ be a matroid on the same ground set $V$. We define the \textit{intersection complex} $\intersectioncomplex(\cC, \cM)$ as the abstract simplicial complex whose vertex set is the collection of all simplices $\sigma \in \cC$ that contain a basis of $\cM$, and where $\sigma_1, \ldots, \sigma_{\ell}$ form a simplex of this complex whenever $\sigma_1 \subset \cdots \subset \sigma_{\ell}$. Also define $\intersectioncomplex(\cC, \cM; k) \coloneqq \intersectioncomplex(\cC, \cM_{\le k})$. It is easy to see that $\intersectioncomplex(\cC, \cM; k)$ is connected if and only if $\recongraph(\cC, \cM; k)$ is connected. The following is our result.

\begin{thm} \label{thm:complex-matroid-connectedness}
Let $\cC$ be a complex and let $\cM$ be a matroid on the same ground set $V$, and let $1 \le k \le r(\cM)$ and $m \ge 0$ be integers. If
\begin{align*}
	\eta_H(\cC[X]) + r(\cM[V - X]) \ge k + m
\end{align*}
for all $X \subseteq V$ for which $V - X$ is a flat of $\cM$ of rank at most $k-1$, then $\eta_H(\intersectioncomplex(\cC, \cM; k)) \ge m+1$. 
\end{thm}

Again, in the case $k = r(\cM)$, the condition of Theorem \ref{thm:complex-matroid-connectedness} can be rephrased as
\begin{align*}
	\eta_H(\cC[X]) \ge r(\cM/(V - X)) + m
\end{align*}
for all nonempty $X \subseteq V$ for which $V - X$ is a flat of $\cM$.
Our proof of Theorem \ref{thm:complex-matroid-connectedness} follows the same approach as Aharoni and Berger's proof of Theorem \ref{thm:aharoni-berger}, which itself was modeled after Welsh's proof \cite{welsh1970matroid} of Edmonds' matroid intersection theorem \cite{edmonds2003submodular} using Rado's theorem \cite{rado1942theorem}. In our new setting, we also need an extra ingredient from the theory of partially ordered sets (or posets). 

Given a finite poset $P = (P, \le)$, its \textit{order complex} $\ordercomplex(P)$ is the abstract simplicial complex with vertex set $P$, in which elements $a_1, \ldots, a_{\ell} \in P$ form a simplex in the complex whenever they form a chain $a_1 \le \cdots \le a_{\ell}$. Both the colorful complex and intersection complex we have defined can be described as order complexes, as noted in the proof below. A \textit{closed interval} of $P$ is a subset of the form $[a, b] \coloneqq \{c \in P : a \le c \le b\}$, for some $a, b \in P$ with $a \le b$. The \textit{interval subdivision} $\mathrm{in}(P)$ of poset $P$ is the poset on the closed intervals of $P$, partially ordered by inclusion $\subseteq$. Notice that $[a_1, b_1] \subseteq [a_2, b_2]$ if and only if $a_2 \le a_1 \le b_1 \le b_2$. We need the following lemma of Walker \cite{walker1988canonical}.

\begin{lem}[\cite{walker1988canonical}] \label{lem:interval-subdivision}
For a finite poset $P$, the order complex $\ordercomplex(\mathrm{in}(P))$ is a simplicial subdivision of the order complex $\ordercomplex(P)$. In particular, $\ordercomplex(\mathrm{in}(P))$ and $\ordercomplex(P)$ have homeomorphic geometric realizations.
\end{lem}

\begin{proof}[Proof of Theorem \ref{thm:complex-matroid-connectedness}]
It suffices to prove the theorem for $k = r(\cM)$, with the result for general $k$ obtained by replacing $\cM$ by its $k$-truncation $\cM_{\le k}$. Say that $V \coloneqq \{v_1, \ldots, v_n\}$. Let $V^{(1)}, V^{(2)}$ denote two disjoint copies of $V$, let $\cC_1$ denote the copy of $\cC$ on $V^{(1)}$, and let $\cM_2$ denote the copy of $\cM$ on $V^{(2)}$. For every vertex $v \in V$, let $v^{(1)}, v^{(2)}$ denote the copies of $v$ in $V^{(1)}, V^{(2)}$, and extend this to copies $X^{(1)}, X^{(2)}$ of vertex subsets $X \subseteq V$. Define the complex $\cD \coloneqq \cC_1 \ast \cM_2^{\ast}$ on vertex set $W \coloneqq V^{(1)} \cup V^{(2)}$, and define the partition $\cW \coloneqq \{W_1, \ldots, W_n\}$ of $W$ by setting $W_i \coloneqq \{v_i^{(1)}, v_i^{(2)}\}$ for all $i$. As usual, we denote $W_I \coloneqq \bigcup_{i \in I} W_i$.

\begin{claim}
We have $\eta_H(\cD[W_I]) \ge |I| + m$ for all nonempty $I \subseteq [n]$.
\end{claim}

\begin{proof}
Fix any nonempty $I \subseteq [n]$, and let $X \coloneqq \{v_i : i \in I\} \subseteq V$. By Propositions \ref{prop:joins} and \ref{prop:matroid-connectedness}, we have
\begin{align*}
	\eta_H(\cD[W_I]) &\ge \eta_H\left(\cC_1[X^{(1)}]\right) + \eta_H\left(\cM_2^{\ast}[X^{(2)}]\right) \\ 
	&\ge \eta_H(\cC[X]) + r(\cM[V - X]) - r(\cM) + |X|.
\end{align*}
First suppose that $V - X$ is not a flat of $\cM$. Then there exists $x \in X$ such that $r(\cM[(V - X) \cup \{x\}]) = r(\cM[V - X])$. This means that $x$ is a loop of $\cM/(V - X)$ and hence is a coloop of $\cM^{\ast}[X]$. In particular, $\eta_H \left(\cM_{2}^{\ast}[X^{(2)}]\right) = \infty$, so that $\eta_H(\cD[W_I]) = \infty$. Now suppose that $V - X$ is a flat of $\cM$. Then combining the above lower bound on $\eta_H(\cD[W_I])$ with the theorem assumption and the fact that $|X| = |I|$, we obtain that $\eta_H(\cD[W_I]) \ge |I| + m$, as required.
\end{proof}

Let $P(\cC, \cM)$ denote the poset of simplices of $\cC$ that contain a basis of $\cM$, partially ordered by inclusion. Also let $P(\cD, \cW)$ denote the poset of simplices of $\cD$ that span the classes of $\cW$, partially ordered by inclusion. Notice that $\intersectioncomplex(\cC, \cM) = \ordercomplex(P(\cC, \cM))$ and $\colorfulcomplex(\cD, \cW) = \ordercomplex(P(\cD, \cW))$.

\begin{claim}
The poset $P(\cD, \cW)$ is order isomorphic to the interval subdivision of the poset $P(\cC, \cM)$.
\end{claim}

\begin{proof}
Define the map $f : P(\cD, \cW) \rightarrow \mathrm{in}(P(\cC, \cM))$ which sends $S^{(1)} \cup T^{(2)}$ to the closed interval $[V - T, S]$. First we verify that $f$ is a well-defined bijection. If $S^{(1)} \cup T^{(2)}$ lies in $P(\cD, \cW)$, then we have $S \in \cC$, $T \in \cM^{\ast}$, and $V - T \subseteq S$. These imply that $V - T$ and $S$ are simplices of $\cC$ that contain a basis of $\cM$, and thus $f(S^{(1)} \cup T^{(2)}) = [V - T, S]$ is indeed a closed interval of $P(\cC, \cM)$. Converely, the inverse $f^{-1} : \mathrm{in}(P(\cC, \cM)) \rightarrow P(\cD, \cW)$ sends $[T, S]$ to $S^{(1)} \cup (V - T)^{(2)}$. If $[T, S]$ is a closed interval of $P(\cC, \cM)$, then $S, T \in \cC \cap \cM^{\ast}$ and $T \subseteq S$, which together imply that $f^{-1}([T, S]) = S^{(1)} \cup (V - T)^{(2)}$ indeed lies in $P(\cD, \cW)$. Now we verify that $f$ is order-preserving. The inclusion $S_1^{(1)} \cup T_1^{(2)} \subseteq S_2^{(1)} \cup T_2^{(2)}$ in $P(\cD, \cW)$ holds if and only if we have $V - T_2 \subseteq V - T_1 \subseteq S_1 \subseteq S_2$, and this holds if and only if we have the inclusion $[V - T_1, S_1] \subseteq [V - T_2, S_2]$ in $P(\cC, \cM)$, i.e., $f(S_1^{(1)} \cup T_1^{(2)}) \subseteq f(S_2^{(1)} \cup T_2^{(2)})$. This proves the claim.
\end{proof}

From Claim 1 and Theorem \ref{thm:topological-reconfiguration-complex}, we get that $\eta_H(\colorfulcomplex(\cD,\cW)) \ge m+1$. From Claim 2 and Lemma \ref{lem:interval-subdivision}, we get that $\intersectioncomplex(\cC, \cM)$ is homeomorphic to $\colorfulcomplex(\cD,\cW)$ as geometric realizations. Combining these facts yields that $\eta_H(\intersectioncomplex(\cC, \cM)) \ge m+1$, as required.
\end{proof}

We remark that basic changes to the proof above can be used to establish the following stronger version of Theorem \ref{thm:complex-matroid-connectedness}, with homology coefficients lying in any fixed field. The result in turn also holds with coefficients in any fixed abelian group, via the universal coefficient theorem. 

\begin{thm} \label{thm:complex-matroid-homology}
Let $\cC$ be a complex and let $\cM$ be a matroid on the same ground set $V$, and let $m \ge 0$ be an integer. If
\begin{align*}
	\widetilde{H}_{r(\cM/(V - X)) + m - 2}(\cC[X]) = 0
\end{align*}
for all nonempty $X \subseteq V$ for which $V - X$ is a flat of $\cM$, then $\widetilde{H}_{m-1}(\intersectioncomplex(\cC, \cM)) = 0$. 
\end{thm}

Finally, we note that Theorem \ref{thm:complex-matroid-connectedness} implies the following deficiency version of our generalized topological Hall theorem (Theorem \ref{thm:topological-reconfiguration-complex}). We define the complex $\colorfulcomplex(\cC, \cV; k)$ to have vertex set consisting of all simplices $\sigma \in \cC$ that span at least $k$ classes of $\cV$, and where $\sigma_1, \ldots, \sigma_\ell$ form a simplex whenever $\sigma_1 \subset \cdots \subset \sigma_\ell$. That is, $\colorfulcomplex(\cC, \cV; k) \coloneqq \intersectioncomplex(\cC, \cM_{\cV}; k)$ where $\cM_{\cV}$ is the partition matroid associated with partition $\cV$. 

\begin{thm} \label{thm:topological-hall-deficiency}
Let $\cC$ be a complex, let $\cV = \{V_1, \ldots, V_n\}$ be a partition of $V(\cC)$, and let $m, d \ge 0$ be integers. If
\begin{align*}
	\eta_H (\cC [ V_I ] ) \ge |I| - d + m \hspace{25pt} \text{for all nonempty } I \subseteq [n] \text{ with } |I| \ge d,
\end{align*}
then $\eta_H(\colorfulcomplex(\cC, \cV; n - d)) \ge m+1$.
\end{thm}

\section{Graph theory applications} \label{sec:combinatorial-applications}

The rest of this paper is dedicated to combinatorial applications of Theorem \ref{thm:topological-reconfiguration-graph}, our reconfiguration variation of the topological Hall theorem. In this section, we describe applications of Theorem \ref{thm:topological-reconfiguration-graph} to reconfiguration problems on graphs and hypergraphs. One could derive more general statements about the higher dimensional homological connectedness of the associated solution spaces using Theorem \ref{thm:topological-reconfiguration-complex}, but we leave these out for brevity.

Some applications rely on the following deficiency version of Theorem \ref{thm:topological-reconfiguration-graph}. This is the special case $m=1$ of Theorem \ref{thm:topological-hall-deficiency}, although it can also be proven more directly from Theorem \ref{thm:topological-reconfiguration-graph} by a standard argument of adding dummy vertices. For a complex $\cC$, a partition $\cV$ of $V(\cC)$, and an integer $k \ge 0$, the reconfiguration graph $\recongraph(\cC, \cV; k)$ has vertex set consisting of all partial colorful simplices of $(\cC, \cV)$ of cardinality $k$, and two such partial colorful simplices are joined by an edge if their union is a simplex in $\cC$ of cardinality $k+1$.

\begin{thm} \label{thm:topological-reconfiguration-deficiency}
Let $\cC$ be a complex, let $\cV = \{V_1, \ldots, V_n\}$ be a partition of $V(\cC)$, let $d \ge 0$ be an integer, and let $\eta \in \{\eta_{\pi}, \eta_H\}$. If
\begin{align*}
	\eta (\cC [ V_I ] ) \ge |I| - d + 1 \hspace{25pt} \text{for all nonempty } I \subseteq [n],
\end{align*}
then $\recongraph(\cC, \cV; n - d)$ is connected.
\end{thm}

\subsection{Reconfigurations of independent transversals in graphs}

In this subsection, we give applications about reconfigurations of independent transversals in graphs. This topic was initiated by Buys, Kang, and Ozeki \cite{buys2025reconfiguration}. We give an alternative proof of their Theorem \ref{thm:BKO}, as well as some variations of it. Recall that an independent transversal of $(G, \cV)$ is a colorful simplex of $(\cI(G), \cV)$, where $\cI(G)$ is the \textit{independence complex} of $G$ (collection of all independent sets of $G$).

First, we deduce domination-type sufficient conditions for reconfigurability. Recall from Theorem \ref{thm:domination-connectedness} that $\tilde{\gamma}(G)$ denotes the minimum size of a strongly dominating set of $V(G)$, and that $i\gamma(G)$ denotes the minimum integer $\ell$ such that every independent set of $G$ is strongly dominated by a vertex set of size at most $\ell$. Combining Theorem \ref{thm:topological-reconfiguration-graph} and Theorem \ref{thm:domination-connectedness}, we obtain the following purely combinatorial sufficient conditions for the reconfiguration graph on independent transversals to be connected. 

\begin{thm} \label{thm:domination}
Let $G$ be a graph, and let $\cV = \{V_1, \ldots, V_n\}$ be a partition of $V(G)$. The following are sufficient conditions for $\recongraph(\cI(G), \cV)$ to be connected:
\begin{itemize}
	\item[(a)] $\tilde{\gamma}(G[V_I]) \ge 2|I|+1$ for all nonempty $I \subseteq [n]$;
	\item[(b)] $i\gamma(G[V_I]) \ge |I|+1$ for all nonempty $I \subseteq [n]$,
\end{itemize}
\end{thm}

These sufficient conditions are excess versions of previous results about the existence of independent transversals, namely Haxell's theorem \cite{haxell1995condition} and the Aharoni--Haxell theorem \cite{aharoni2000hall}, respectively. Theorem \ref{thm:domination}(a) leads to the following maximum degree condition for the reconfigurability of independent transversals, which is also a corollary of Theorem \ref{thm:BKO} by Buys, Kang, and Ozeki \cite{buys2025reconfiguration}.

\begin{cor}\label{cor:two-Delta-plus-one}
Let $G$ be a graph, and let $\cV = \{V_1, \ldots, V_n\}$ be a partition of $V(G)$. If $G$ has maximum degree $\Delta$ and $|V_i| \ge 2\Delta + 1$ for all $i$, then $\recongraph(\cI(G), \cV)$ is connected.
\end{cor}

\begin{proof}
Fix any nonempty $I \subseteq [n]$, so that $|V_I| \ge (2\Delta+1)|I|$. Since every vertex in $V_I$ has at most $\Delta$ neighbors in $G[V_I]$, there cannot be a totally dominating set in $G[V_I]$ with cardinality less than $\frac{|V_I|}{\Delta}$, so that $\tilde{\gamma}(G[V_I]) \ge \frac{|V_I|}{\Delta} \ge \frac{(2\Delta+1)|I|}{\Delta} > 2|I|$. Applying Theorem \ref{thm:domination}(a), we deduce that $\recongraph(\cI(G), \cV)$ is connected.
\end{proof}

Now, Haxell's theorem \cite{haxell1995condition, haxell2001note} states that an independent transversal exists whenever we have $\tilde{\gamma}(G[V_I]) \ge 2|I|-1$ for all nonempty $I \subseteq [n]$. This can likewise be derived by combining the standard topological Hall theorem (Theorem \ref{thm:topological-hall}) and Theorem \ref{thm:domination-connectedness}, as basically done in \cite{aharoni2002triangulated, meshulam2001clique}. On the other hand, Haxell's original proof \cite{haxell1995condition} was purely combinatorial, described on line graphs of hypergraphs but also observed to work on general graphs (see \cite{aharoni2007independent, haxell2001note}).

Our next goal is to give a similar purely combinatorial proof of Theorem \ref{thm:domination}(a), not using topological methods. In turn, this yields a simplified combinatorial proof of Corollary \ref{cor:two-Delta-plus-one} compared to that of Buys, Kang, and Ozeki \cite{buys2025reconfiguration}. Our proof ideas are similar to those of \cite{buys2025reconfiguration}, but they incorporate a more domination emphasis like in the combinatorial proof of Haxell's theorem described in \cite{aharoni2007independent}.

For notation, given two vertex subsets $X,Y \subseteq V(G)$, their symmetric difference is the set $X \triangle Y \coloneqq (X - Y) \cup (Y - X)$. Given vertex subset $X \subseteq V(G)$, we denote $I(X) \coloneqq \{i \in [n] : v \in V_i \text{ for some } v \in X\}$. Given a transversal $T$ of classes $\{V_j : j \in J\}$ and a vertex $v \in V_j$ for some $j \in J$, we denote $T \oplus v \coloneqq (T - V_j) \cup \{v\}$, which is the transversal obtained from $T$ by replacing the representative of $V_j$ by the vertex $v$. 

\begin{proof}[Alternative proof of Theorem \ref{thm:domination}(a)]
We will prove the contrapositive statement: If $\recongraph(\cI(G), \cV)$ is disconnected, then there exists a nonempty index set $I \subseteq [n]$ and a vertex set $D \subseteq V_I$ such that $|D| \le 2|I|$ and $D$ totally dominates $G[V_I]$. First, if an independent transversal does not exist, then Haxell's theorem \cite{haxell1995condition, haxell2001note} implies the statement with $|D| \le 2|I| - 2$. Thus, we assume from now on that $\recongraph(\cI(G), \cV)$ is nonempty and disconnected. 

Let $C_1$ and $C_2$ be any two distinct connected components of $\recongraph(\cI(G), \cV)$. We pick an independent transversal from each of these components, say $S$ from $C_1$ and $T$ from $C_2$, in such a way that $I(S \triangle T)$ has minimum cardinality. By definition, $I(S \triangle T)$ is nonempty. We start with the index set $I_0 \coloneqq I(S \triangle T)$ and vertex set $D_0 \coloneqq S \triangle T \subseteq V_{I_0}$, noting that $|D_0| = 2|I_0|$. We also put $R_0 \coloneqq S \cap T$.

Inductively for $k \ge 1$, assume that we have constructed an index set $I_{k-1} \subseteq [n]$, a vertex set $D_{k-1} \subseteq V_{I_{k-1}}$ with $|D_{k-1}| \le 2|I_{k-1}|$, and an independent transversal $R_{k-1}$ of the classes $\{V_i : i \in I(S \cap T)\}$ such that $(S - T) \cup R_{k-1}$ is an independent transversal in $C_1$ and such that $(T - S) \cup R_{k-1}$ is an independent transversal in $C_2$. If $D_{k-1}$ totally dominates $G[V_{I_{k-1}}]$, then we simply return $I \coloneqq I_{k-1}$ and $D \coloneqq D_{k-1}$, and the theorem is proved. Otherwise, we first select any vertex $x_k$ in $V_{I_{k-1}}$ that is not totally dominated by $D_{k-1}$. Among all independent transversals $R_k$ of the classes $\{V_i : i \in I(S \cap T)\}$ with the properties that
\begin{itemize}
	\item $R_{k}$ agrees with $R_{k-1}$ on $I_{k-1} - I(S \triangle T)$,
	\item $(S - T) \cup R_{k}$ is an independent transversal in $C_1$, and
	\item $(T - S) \cup R_{k}$ is an independent transversal in $C_2$,
\end{itemize}
choose one so that $x_k$ has the minimum number of neighbors in $R_k$. Such an $R_k$ exists because $R_{k-1}$ satisfies the three stated properties. Let $Y_k \coloneqq N(x_k) \cap R_k$. Assuming that $Y_k$ is nonempty, we then put $I_{k} \coloneqq I_{k-1} \cup I(Y_k)$ and $D_k \coloneqq D_{k-1} \cup \{x_k\} \cup Y_k$, giving us that $D_k \subseteq V_{I_{k}}$ and $|D_{k}| \le 2|I_k|$. The procedure may then continue. Thus, assuming that the sets $Y_1, Y_2, \ldots$ are all nonempty, we get that $I_0, I_1, I_2, \ldots$ is a growing sequence of index sets contained in $[n]$, meaning the procedure eventually terminates and the statement is proved.

What is left to show is that at each step $k \ge 1$, the vertex set $Y_k$ is indeed nonempty. Suppose for contradiction that $Y_k = \emptyset$, and that we have chosen the smallest such step $k$. Recall that $x_k \in V_{I_{k-1}}$ is not totally dominated by $D_{k-1}$. First suppose that $x_k \in V_{I(S \triangle T)}$. Then we may update $S$ to $S' \coloneqq S \oplus x_k$ and update $T$ to $T' \coloneqq T \oplus x_k$, and now $I(S' \triangle T')$ has smaller cardinality than $I(S \triangle T)$, contradicting how $S$ and $T$ were originally chosen. We cannot have $x_k \in D_{k-1} - (S \triangle T)$ because, by construction, the vertex set $D_{k-1} - (S \triangle T)$ induces vertex-disjoint stars each on at least two vertices. Hence, we have $x_k \in V_j$ for some $j \in I_{k-1} - I(S \triangle T)$ and with $x_k \notin D_{k-1}$. But then letting $\ell \le k-1$ be the smallest index for which $j \in I_{\ell}$, we could have replaced $R_{\ell}$ by $R' = R_{\ell} \oplus x_{k}$, and then $x_{\ell}$ has fewer neighbors in $R'$ than it did in $R_{\ell}$, contradicting the choice of $R_{\ell}$ at step $\ell$. This finishes the proof.
\end{proof}

We proceed to give a complete alternative proof of Theorem \ref{thm:BKO} by Buys, Kang, and Ozeki \cite{buys2025reconfiguration} using the topological approach. Recall that Corollary \ref{cor:two-Delta-plus-one} was a direct consequence of Theorem \ref{thm:topological-reconfiguration-graph} and the topological connectedness lower bound 
\begin{align*}
	\eta(\cI(G)) \ge \frac{|V(G)|}{2\Delta}.
\end{align*}
This latter lower bound was a combination of the inequalities $\eta(\cI(G)) \ge \frac{\tilde{\gamma}(G)}{2}$ and $\tilde{\gamma}(G) \ge \frac{|V(G)|}{\Delta}$. To fully prove Theorem \ref{thm:BKO}, we also need a characterization of when the lower bound $\eta(\cI(G)) \ge \frac{|V(G)|}{2\Delta}$ is tight. This is provided by the following lemma, which was implicit in \cite{aharoni2015cooperative} and proven explicitly in \cite{haxell2024degree}. 

\begin{lem}[\cite{haxell2024degree}] \label{lemma:connectedness-maxdegree-1}
	If $G$ is a graph with maximum degree $\Delta$, then $\eta(\cI(G)) \ge \frac{|V(G)|}{2\Delta}$ with equality if and only if $G$ is the disjoint union of $\frac{|V(G)|}{2\Delta}$ copies of the complete bipartite graph $K_{\Delta,\Delta}$.
\end{lem}

\begin{proof}[Proof of Theorem \ref{thm:BKO}]
Let $G$ be a graph and let $\cV = \{V_1, \ldots, V_n\}$ be a partition of $V(G)$ satisfying the hypothesis of the theorem. Fix any nonempty subset $I \subseteq [n]$, so that $|V_I| \ge 2\Delta|I|$. By assumption, $G[V_I]$ is not the disjoint union of $\frac{|V_I|}{2\Delta}$ copies of $K_{\Delta,\Delta}$, so from Lemma \ref{lemma:connectedness-maxdegree-1} we deduce that $\eta(\cI(G[V_I])) > \frac{|V_I|}{2\Delta} \ge |I|$. Since $\eta(\cdot)$ is integral, this implies that $\eta(\cI(G[V_I])) \ge |I|+1$. By Theorem \ref{thm:topological-reconfiguration-graph}, we conclude that $\recongraph(\cI(G),\cV)$ is connected.
\end{proof}

The sufficient condition for $\recongraph(\cI(G),\cV)$ being connected given in Theorem \ref{thm:BKO} is tight. One easy example that demonstrates this is a single complete bipartite graph $K_{\Delta, \Delta}$ with a single class $V_1$. Another tight example is two disjoint copies of $K_{\Delta,\Delta}$ and two classes $V_1,V_2$ forming standard bipartitions on both of the components. These have disconnected reconfiguration graphs. All tight examples for Theorem \ref{thm:BKO} are precisely characterized in \cite{wdowinski2025tight}, where they are shown to be generated by a simple constructive procedure. On the other hand, for many classes of graphs which avoid copies of $K_{\Delta,\Delta}$, it is possible to decrease the class size lower bound $2\Delta+1$. The following theorem lists examples of this. These are obtained by combining Theorem \ref{thm:topological-reconfiguration-graph} with a topological connectedness lower bound given in the adjacent reference.

\begin{thm} \label{thm:maximum-degree-special-1}
	Let $G$ be a graph with maximum degree $\Delta$ and let $\cV = \{V_1, \ldots, V_n\}$ be a partition of $V(G)$. The following are sufficient conditions for $\recongraph(\cI(G), \cV)$ to be connected:
	\begin{itemize}
		\item[(a)] $G$ is a chordal graph and $|V_i| > \Delta + 1$ for all $i$. \cite{aharoni2002tree}
		\item[(b)] $G$ does not contain the star $K_{1,k+1}$ as an induced subgraph and $|V_i| > \frac{(2k - 1)\Delta + k}{k}$ for all $i$. \cite{aharoni2016eigenvalues}
		\item[(c)] $G$ is the line graph of a $k$-uniform linear hypergraph and $|V_i| > \Delta + k$ for all $i$. \cite{aharoni2016eigenvalues}
		\item[(d)] $G$ has maximum average degree $a$ and $|V_i| > \frac{2\Delta^2}{2\Delta - a}$ for all $i$. \cite{haxell2024degree}
	\end{itemize}
\end{thm}

\subsection{Reconfigurations of rainbow matchings}

In this subsection, we describe two applications about reconfigurations of rainbow matchings in graphs and hypergraphs. Given a multi-hypergraph $G$ and a partition $\cE = \{E_1, \ldots, E_n\}$ of its edge set $E(G)$, a \textit{rainbow matching} of $(G, \cE)$ is a matching of $G$ that forms a partial transversal of $\cE$. It is a \textit{full rainbow matching} if it is a transversal of $\cE$. Letting $\cM(G)$ denote the matching complex of $G$, a full rainbow matching of $(G, \cE)$ is the same as a colorful simplex of $(\cM(G), \cE)$, and thus we study the reconfiguration graph $\recongraph(\cM(G), \cE)$. The existence of rainbow matchings is a widely studied topic inspired by the famous Ryser--Brualdi--Stein conjecture \cite{brualdi1991combinatorial, ryser1967neuere, stein1975transversals} (a proof for large even $n$ was announced in \cite{montgomery2023proof}), which asserts the existence of a rainbow matching of size $n - 1$ in any edge partition of $K_{n,n}$ into $n$ perfect matchings, as well as a full rainbow matching when $n$ is odd. 
Our work motivates the study of reconfigurations of rainbow matchings. 

For one result on reconfigurations of rainbow matchings, we have the following anologue of Corollary \ref{cor:two-Delta-plus-one} in the edge setting. 

\begin{thm} \label{thm:full-rainbow}
	Let $G$ be an $r$-graph, and let $\cE = \{E_1, \ldots, E_n\}$ be a partition of $E(G)$. If $G$ has maximum degree $\Delta$ and $|E_i| \ge r\Delta+1$ for all $i$, then $\recongraph(\cM(G), \cE)$ is connected.
\end{thm}

This follows from combining Theorem \ref{thm:topological-reconfiguration-graph} and the topological connectedness lower bound $\eta_H(\cM(G)) \ge \frac{|E(G)|}{r\Delta}$. The latter inequality is derived from Theorem \ref{thm:matching-number} by noticing that the vector $(x_e)_{e \in E(G)}$ given by $x_e = 1/\Delta$, for all $e \in E(G)$, is a fractional matching of $G$, so that $\nu^\ast(G) \ge |E(G)|/\Delta$. The sufficient condition of Theorem \ref{thm:full-rainbow} is tight. The easiest examples with $|E_i| = r\Delta$ for all $i$ and $\recongraph(\cM(G), \cE)$ disconnected are $r \times r$ grids $G$ with a single class $E_1$. The $r \times r$ grid $G$ consists of $r^2$ vertices $\{v_{i,j} : i, j \in [r]\}$, and there $2r$ edges $e_k = \{v_{k, j} : j \in [r]\}$ and $f_k = \{v_{i,k} : i \in [r]\}$ for $k \in [r]$. The maximum degree here is $\Delta = 2$. For tight examples with larger maximum degree $\Delta$, we may blow up the edges of $G$ into many parallel edges. More interesting tight examples can be produced using the construction procedure of \cite{wdowinski2025tight}.

In the direction of the Ryser--Buraldi--Stein conjecture, a more general conjecture of Stein \cite{stein1975transversals} asserts the existence a rainbow matching of size $n - 1$ in $K_{n,n}$ whenever $E_1, \ldots, E_n$ are any pairwise disjoint edge subsets of size $n$. Instead of a Latin square, such an edge partition corresponds to a so-called equi-$n$-square. Stein's conjecture was disproven in \cite{pokrovskiy2019counterexample}, but see \cite{anastos2025note, chakraborti2024almost} for more recent results. Aharoni, Berger, Kotlar, and Ziv \cite{aharoni2017conjecture} used topological methods to show that there exists a rainbow matching of size at least $\frac{2n}{3} - \frac{1}{2}$ in the setting of Stein's conjecture. The following is a reconfiguration analogue of their result.

\begin{thm} \label{thm:Latin-square}
	If $\cE = \{E_1, \ldots, E_n\}$ is a partition of the edge set of $K_{n,n}$ into $n$ sets each of size $n$, then $\recongraph(\cM(K_{n,n}), \cE; \left\lceil \frac{2n}{3} - \frac{3}{2} \right\rceil)$ is connected.
\end{thm}

This follows from combining Theorem \ref{thm:topological-reconfiguration-deficiency} and \cite[Theorem 3.10]{aharoni2017conjecture}, the latter being a lower bound on the topological connectedness of the matching complex of subgraphs of $K_{n,n}$. In terms of equi-$n$-squares, Theorem \ref{thm:Latin-square} states that for any two partial transversals of size $\left\lceil \frac{2n}{3} - \frac{3}{2} \right\rceil$ in an equi-$n$-square, there exists a sequence of partial transversals from one to other where a given partial transversal in the sequence is obtained from the previous one by adding an entry in a distinct row and column from the other entries of that partial transversal, and then deleting some other entry. We leave open the following problems, for which the corresponding existence problems were solved in \cite{chakraborti2024almost} and \cite{montgomery2023proof}, respectively.

\begin{prob} \label{prob:equi}
If $\cE = \{E_1, \ldots, E_n\}$ consists of edge sets $E_i$ each of size $n$, estimate the optimal threshold $k = k(n)$ for which $\recongraph(\cM(K_{n,n}), \cE; k)$ is connected. In particular, determine if $k(n) = n - o(n)$ is sufficient. 
\end{prob}

\begin{prob} \label{prob:latin}
If $\cE = \{E_1, \ldots, E_n\}$ consists of perfect matchings $E_i$, determine if $k(n) = n - c$ is sufficient for  $\recongraph(\cM(K_{n,n}), \cE; k)$ being connected, for some absolute constant $c \ge 2$.
\end{prob}

\subsection{Reconfigurations of bipartite hypergraph matchings} \label{sec:hypergraph-matching}

In this subsection, we describe applications about reconfigurations of matchings in bipartite hypergraphs. Reconfigurations of matchings in graphs have been a much studied topic \cite{bonamy2019perfect, ito2011complexity, ito2019reconfiguration}. As one of the first results in this area, Ito, Demaine, Harvey, Papadimitriou, Sideri, Uehara, and Uno \cite{ito2011complexity} showed that the following natural matching reconfiguration problem can be solved in polynomial-time: given two matchings of a graph each of size $k$, determine whether one can be obtained from the other by deleting or adding one edge at a time while always maintaining a matching with at least $k - 1$ edges. The matching reconfiguration problem on bipartite graphs that we define below is basically equivalent to this one, but for bipartite hypergraphs with larger edge sizes our adjacency conditions are slightly stronger.

A multi-hypergraph $H$ is said to be \textit{bipartite} if its vertex set can be partitioned into two classes $(A,B)$, such that every edge of $H$ intersects $A$ in exactly one vertex. For a subset $X \subseteq A$, its \textit{link} $\mathrm{lk}_H(X)$ is the multi-hypergraph with vertex set $B$ and edge multiset $\{e - \{x\} : e \in E(H), x \in e \cap A\}$. Observe that a matching of $H$ that covers a subset $X \subseteq A$ corresponds to a full rainbow matching of $(G, \cE)$ where $G \coloneqq H[B]$ and $\cE \coloneqq \{ \mathrm{lk}_H(x) : x \in X\}$. We define the reconfiguration graph $\recongraph_{\mathrm{Mat}}(H, A; k)$ as having vertex set consisting of all matchings of $H$ of size $k$, and two such matchings $M_1, M_2$ are joined by an edge in the graph if their symmetric difference consists of two edges that are non-intersecting on the $B$-side. When $k = |A|$, we denote the reconfiguration graph by $\recongraph_{\mathrm{Mat}}(H, A)$. 

Using the above connection to rainbow matchings, Theorem \ref{thm:topological-reconfiguration-deficiency} translates to the following result in the setting of bipartite hypergraph matchings.

\begin{thm} \label{thm:deficiency-link}
Let $H$ be a bipartite multi-hypergraph with bipartition $(A,B)$, let $d \ge 0$ be an integer, and let $\eta \in \{\eta_{\pi}, \eta_H\}$. If
\begin{align*}
	\eta(\cM(\mathrm{lk}_H(X))) \ge |X| - d + 1 \hspace{25pt} \text{ for all nonempty } X \subseteq A,
\end{align*}
then $\recongraph_{\mathrm{Mat}}(H, A; |A| - d)$ is connected.
\end{thm}

Combining Theorem \ref{thm:deficiency-link} (for $d = 0$) and Theorem \ref{thm:matching-number}, we immediately get a reconfiguration version of Hall's theorem for hypergraphs by Aharoni and Haxell \cite{aharoni2000hall}. Recall that $\nu(\cdot)$ denotes the matching number of a multi-hypergraph (it could also be replaced by the fractional matching number $\nu^\ast(\cdot)$ here).

\begin{thm} \label{thm:hall-hypergraph}
Let $H$ be a bipartite $r$-graph with bipartition $(A,B)$. If
\begin{align*}
	\nu(\mathrm{lk}_H(X)) \ge (r - 1)|X|+1 \hspace{25pt} \text{ for all nonempty } X \subseteq A,
\end{align*}
then $\recongraph_{\mathrm{Mat}}(H, A)$ is connected.
\end{thm}

The condition of Theorem \ref{thm:hall-hypergraph} is tight. Examples showing this are similar to those of Theorem \ref{thm:full-rainbow}. Specifically, construct a bipartite hypergraph $H$ by starting with the $(r-1) \times (r-1)$ grid and extending each edge by adding a single vertex $x$. Then $\nu(\mathrm{lk}_H(\{x\})) = r - 1$ but $\recongraph_{\mathrm{Mat}}(H, \{x\})$ is disconnected. More interesting tight examples are also possible. We highlight Theorem \ref{thm:hall-hypergraph} in the special case of graphs, $r = 2$. 

\begin{cor} \label{cor:hall-reconfiguration}
Let $H$ be a bipartite graph with bipartition $(A,B)$. If
\begin{align*}
	|N(X)| \ge |X| + 1 \hspace{25pt} \text{for all nonempty } X \subseteq A,
\end{align*}
then $\recongraph_{\mathrm{Mat}}(H, A)$ is connected
\end{cor}

Combining the deficiency version of Corollary \ref{cor:hall-reconfiguration} with Hall's matching theorem, we deduce the following.

\begin{cor} \label{cor:konig-reconfiguration}
If $H$ is a bipartite graph with bipartition $(A,B)$ and $k \le \nu(H) - 1$, then $\recongraph_{\mathrm{Mat}}(H, A; k-1)$ is connected.
\end{cor}

Corollaries \ref{cor:hall-reconfiguration} and \ref{cor:konig-reconfiguration} are not exactly new, as they are easy consequences of the matching reconfigurability criterion given in \cite[Lemma 1]{ito2011complexity}. As mentioned above, there is also a polynomial-time algorithm for finding appropriate reconfiguration sequences of matchings. In their monograph, Lov\'asz and Plummer \cite{lovasz2009matching} refer to the quantity $\max \{|N(X)| - |X| : X \subseteq A, X \neq \emptyset\}$ as the \textit{surplus} of bipartite graph $H$ with bipartition $(A, B)$, and they give some combinatorial characterizations of it. However, the surplus does not appear to have previously been related to matching reconfigurations. Unlike the case $r = 2$, the reconfiguration graph $\recongraph_{\mathrm{Mat}}(H, A; \nu(H)-1)$ is not necessarily connected for bipartite $r$-graphs $H$ when $r \ge 3$. For example, consider the $r$-graph $H$ with the $2r$ vertices $x_1, \ldots, x_r, y_1, \ldots, y_r$ and the $4$ edges $\{x_1, x_2, x_3, \ldots, x_r\}$, $\{x_1, x_2, y_3, \ldots, y_r\}$, $\{y_1, y_2, x_3, \ldots, x_r\}$, $\{y_1, y_2, y_3, \ldots, y_r\}$, and let $A = \{x_1, y_1\}$. Observe that the edge $\{x_1, x_2, x_3, \ldots, x_r\}$ cannot be reconfigured to the edge $\{x_1, x_2, y_3, \ldots, y_r\}$.

We conclude by considering a reconfiguration version of Ryser's conjecture \cite{henderson1971permutation} for $r$-graphs. A multi-hypergraph $H$ is said to be \textit{$r$-partite} if there exists a partition of its vertex set into $r$ classes $(A_1, \ldots, A_r)$ such that every edge intersects every class $A_i$ in at most one vertex. Notice that an $r$-partite $r$-graph is necessarily a bipartite multi-hypergraph. A \textit{vertex cover} of a multi-hypergraph $H$ is a set of vertices that intersects every edge. Let $\tau(H)$ denote the minimum size of a vertex cover of $H$. It is easy to derive from definitions that every $r$-graph $H$ satisfies $\nu(H) \le \tau(H) \le r \cdot \nu(H)$. Ryser (actually his PhD student Henderson \cite{henderson1971permutation}) conjectured that the second inequality can be improved to $\tau(H) \le (r - 1)\nu(H)$ whenever $H$ is $r$-partite. K\H{o}nig's theorem is the case $r = 2$ of Ryser's conjecture, and in a breakthrough application of the topological Hall theorem, Aharoni \cite{aharoni2001ryser} proved the case $r = 3$. Tight constructions for $r = 3$ have been characterized in \cite{haxell2018extremal}, while Ryser's conjecture is still wide open for $r \ge 4$. A generalized conjecture by Lov\'asz \cite{lovasz1975minimax} has recently been disproven in \cite{abiad2025infinitely, clow2025counterexample}. Aharoni's theorem has the following reconfiguration variation.

\begin{thm}
Let $H$ be a $3$-partite $3$-graph and let $A$ be one of the $3$ partition classes. If $k < \tau(H)/2$, then $\recongraph_{\mathrm{Mat}}(H, A; k)$ is connected.
\end{thm}

This is proven along the same lines as Aharoni's theorem (see \cite{haxell2019topological}), namely by combining Theorem \ref{thm:deficiency-link}, Theorem \ref{thm:matching-number}, and the following straightforward inequality observed by Aharoni \cite{aharoni2001ryser}:
\begin{align*}
	\nu(\text{lk}_H(X)) \ge |X| - |A| + \tau(H) \hspace{25pt} \text{for all } X \subseteq A.
\end{align*}
In the spirit of Ryser's conjecture, we conjecture that our type of reconfiguration result holds for all larger uniformities.

\begin{conj} \label{conj:ryser}
Let $H$ be an $r$-partite $r$-graph and let $A$ be one of the $r$ partition classes, with $r \ge 2$. If $k < \tau(H)/(r - 1)$, then $\recongraph_{\mathrm{Mat}}(H, A; k)$ is connected.
\end{conj}

It is perhaps more natural to study this conjecture in terms of the usual notion of matching reconfiguration described at the beginning of this subsection, where the adjacency condition is slightly relaxed.

\subsection{Reconfigurations of list colorings} \label{sec:list-coloring}

In this subsection, we explain how our reconfiguration results relate to the reconfiguration of (list) colorings of graphs and hypergraphs. Coloring reconfigurations are some of the most commonly studied reconfiguration problems on graphs (e.g., \cite{bousquet2022polynomial, cereceda2008connectedness, feghali2016reconfigurations, ito2012reconfiguration}, as well as the surveys \cite{mynhardt2019reconfiguration, nishimura2018introduction}).

For list vertex-coloring, let $H$ be a graph and let $L = (L(x) : x \in V(H))$ be a list assignment for $V(H)$. The reconfiguration graph $\recongraph_{\mathrm{VLC}}(H, L)$ has vertex set consisting of all proper $L$-colorings of $H$, and two proper $L$-colorings are joined by an edge if they differ at one vertex's color. Proper $L$-colorings of $H$ can be represented as independent transversals in an auxiliary vertex-partitioned graph $(G, \cV)$ (see \cite{haxell2001note}). The \textit{maximum vertex-color degree} of $(H,L)$ is the maximum, over all colors $c$, of the maximum degree of the subgraph of $H$ induced by vertex subset $\{v \in V(H) : c \in L(v)\}$.
Buys, Kang, and Ozeki \cite{buys2025reconfiguration} observed the following consequence of their Theorem \ref{thm:BKO}.

\begin{cor}[\cite{buys2025reconfiguration}] \label{cor:vertex-coloring}
Let $H$ be a graph and let $L$ be a list assignment for $V(H)$. If $(H, L)$ has maximum vertex-color degree $\Delta$ and $|L(v)| \ge 2\Delta+1$ for all $v \in V(H)$, then $\recongraph_{\mathrm{VLC}}(H, L)$ is connected. 
\end{cor}

Based on the characterization of tight constructions for Theorem \ref{thm:topological-reconfiguration-graph} described in \cite{wdowinski2025tight}, the condition of Corollary \ref{cor:vertex-coloring} can be slightly improved to $|L(v)| \ge 2\Delta$ for all $v$. Buys, Kang, and Ozeki \cite{buys2025reconfiguration} posed the following problem.

\begin{prob}[\cite{buys2025reconfiguration}] \label{prob:vertex-coloring}
Determine an optimal bound $h(\Delta)$ such that $\recongraph_{\mathrm{VLC}}(H, L)$ is connected whenever $(H,L)$ has maximum vertex-color degree $\Delta$ and $|L(v)| \ge h(\Delta)$ for all $v \in V(H)$.
\end{prob}

The existence of a proper $L$-coloring in this setting when $h(\Delta) = \Delta + o(\Delta)$ was proven by Reed and Sudakov \cite{reed2002asymptotically} using probabilistic methods, with later generalizations given for independent transversals in locally sparse settings \cite{glock2022average, kang2022colorings, loh2007independent}. Examples with $|L(v)| = \Delta+1$ for all $v \in V(H)$ and no proper $L$-coloring were described by Bohman and Holzman \cite{bohman2002list} (see also \cite{haxell2024constructing}), which disproved a conjecture of Reed \cite{reed1999list}. We note that there is an easier example where $|L(v)| = \Delta+1$ for all $v \in V(H)$ and the reconfiguration graph $\recongraph_{\mathrm{VLC}}(H, L)$ is disconnected, namely when $H = K_{\Delta+1}$ and $L(v) = \{1, \ldots, \Delta+1\}$ for all $v \in V(H)$. On the other hand, it is easy to show that $\recongraph_{\mathrm{VLC}}(H,L)$ is connected if $H$ itself has maximum degree $\Delta$ and $|L(v)| \ge \Delta+2$ for all $v \in V(H)$ (similar to \cite{cereceda2008connectedness, jerrum1995very}).

For list edge-coloring, given a multi-hypergraph $H$ and a list assignment $L = (L(e) : e \in E(H))$ for $E(H)$, the reconfiguration graph $\recongraph_{\mathrm{ELC}}(H, L)$ has vertex set consisting of all proper $L$-colorings of $H$, and two proper $L$-colorings are joined by an edge if they differ at one edge's color. Analogous to list vertex-colorings, proper list edge-colorings of a multi-hypergraph can be represented as full rainbow matchings in an auxiliary edge-partitioned multi-hypergraph $(G, \cE)$ (see \cite{wdowinski2024bounded}). The \textit{maximum edge-color degree} of $(H,L)$ is the maximum, over all colors $c$, of the maximum degree of subhypergraph of $H$ induced by the edge subset $\{e \in E(H) : c \in L(e)\}$. The following is a consequence of Theorem \ref{thm:full-rainbow}.

\begin{cor}
Let $H$ be an $r$-graph and let $L$ be a list assignment for $E(H)$. If $(H, L)$ has maximum edge-color degree $\Delta$ and $|L(e)| \ge r \Delta+1$ for all $e \in E(H)$, then $\recongraph_{\mathrm{ELC}}(H, L)$ is connected.
\end{cor}

Here, we pose a similar problem to the one above by Buys, Kang, and Ozeki \cite{buys2025reconfiguration}.

\begin{prob} \label{prob:edge-coloring}
Determine an optimal bound $h(\Delta)$ such that $\recongraph_{\mathrm{ELC}}(H, L)$ is connected whenever $(H,L)$ has maximum edge-color degree $\Delta$, $H$ has maximum codegree $o(\Delta)$, and $|L(e)| \ge h(\Delta)$ for all $e \in E(H)$.
\end{prob}

The existence of an $L$-coloring in this setting when $h(\Delta) = \Delta + o(\Delta)$ was shown by Delcourt and Postle \cite{delcourt2022finding} (see also \cite{wdowinski2024bounded}), generalizing Kahn's \cite{kahn1996asymptotically} asymptotic version of the famous List Edge-Coloring Conjecture. This was also done in a more general rainbow matching setting. Examples of bipartite graphs $H$ and list assignments $L$ with $|L(e)| = \Delta$ for all $e$ and no proper $L$-coloring were described in \cite{wdowinski2024bounded}.

\section{Discrete geometry applications} \label{sec:geometric-applications}

In this section, we describe applications of Theorem \ref{thm:topological-reconfiguration-graph} (or more accurately Theorem \ref{thm:complex-matroid-reconfiguration}, its matroidal generalization) to reconfiguration results in discrete geometry. Namely, we prove reconfiguration analogues of the colorful Helly, colorful Carath\'eodory, and Tverberg theorems. In the last subsection, we also mention generalizations of these results to higher dimensional homological connectedness. The above three theorems in their classical forms are recalled below, and see \cite{amenta2017helly, barany2018tverberg, deloera2019discrete, holmsen2017helly} for recent surveys about them.

The colorful Helly theorem was discovered by Lov\'asz, and the colorful Carath\'eodory theorem was subsequently discovered by B\'ar\'any, with both results originally reported in \cite{barany1982generalization}. The two results are related via linear programming duality. The ordinary Helly and Carath\'eodory theorems are recovered by taking all the families or all the point sets to be the same.

\begin{thm}[Colorful Helly theorem \cite{barany1982generalization}] \label{thm:colorful-helly}
Let $\cF_1, \ldots, \cF_{n}$ be finite families of convex subsets of $\mathbb{R}^d$. If $\bigcap \cF_i = \emptyset$ for all $i \in [n]$ and $n \ge d+1$, then there exist $C_i \in \cF_i$ for $i \in [n]$ such that $\bigcap_{i=1}^n C_i = \emptyset$.
\end{thm}

\begin{thm}[Colorful Carath\'eodory theorem \cite{barany1982generalization}] \label{thm:colorful-caratheodory}
Let $A_1, \ldots, A_n$ be finite sets of points in $\mathbb{R}^d$, and let $x \in \mathbb{R}^d$. If $x \in \convex(A_i)$ for all $i \in [n]$ and $n \ge d+1$, then there exist $a_i \in A_i$ for $i \in [n]$ such that $x \in \convex(\{a_1, \ldots, a_n\})$.
\end{thm}

Tverberg's theorem \cite{tverberg1966generalization} is given below, with the case $r = 2$ known before and called Radon's theorem. The original proof by Tverberg was complicated, but simpler proofs were subsequently discovered (see \cite{barany2018tverberg}).

\begin{thm}[Tverberg's theorem \cite{tverberg1966generalization}] \label{thm:tverberg}
Let $X$ be a finite set of points in $\mathbb{R}^d$, and let $r \ge 1$ be an integer. If $|X| \ge (d+1)(r - 1)+1$, then there exists a partition $X_1, \ldots, X_r$ of $X$ into $r$ parts such that $\bigcap_{i=1}^r \convex(X_i) \neq \emptyset$.
\end{thm}

\subsection{Reconfigurations for colorful Helly and colorful Carath\'eodory theorems}

In this subsection, we state and prove our reconfiguration versions of the colorful Helly and the colorful Carath\'eodory theorems. We write these results in terms of reconfigurations of ordered tuples with possible repetitions, rather than transversals, because they translate better to the setting of Tverberg's theorem.

For a reconfiguration version of the colorful Helly theorem by Lov\'asz \cite{barany1982generalization}, we are given finite families $\cF_1, \ldots, \cF_n$ each consisting of convex subsets of $\mathbb{R}^d$. The reconfiguration graph $\recongraph_{\mathrm{CH}}(\cF_1, \ldots, \cF_{n})$ has vertex set consisting of all ordered tuples $(C_1, \ldots, C_n)$ satisfying the properties that $C_i \in \cF_i$ for all $i \in [n]$ and $\bigcap_{i=1}^n C_i = \emptyset$. Two such tuples $(C_1, \ldots, C_n)$ and $(D_1, \ldots, D_n)$ are joined by an edge in the graph whenever there exists a unique index $j \in [n]$ such that $C_j \neq D_j$ and $\bigcap_{i \neq j} C_i = \emptyset$. The colorful Helly theorem is equivalent to the statement that $\recongraph_{\mathrm{CH}}(\cF_1, \ldots, \cF_{n})$ is nonempty whenever $\bigcap \cF_i = \emptyset$ for all $i \in [n]$ and $n \ge d+1$. We show the following.

\begin{thm} \label{thm:colorful-helly-reconfigure}
Let $\cF_1, \ldots, \cF_{n}$ be finite families of convex subsets of $\mathbb{R}^d$. If $\bigcap \cF_i = \emptyset$ for all $i \in [n]$ and $n \ge d+2$, then $\recongraph_{\mathrm{CH}}(\cF_1, \ldots, \cF_{n})$ is connected.
\end{thm}

For a reconfiguration version of the colorful Carath\'eodory theorem by B\'ar\'any \cite{barany1982generalization}, we are given finite sets $A_1, \ldots, A_n$ of points in $\mathbb{R}^d$ and a point $x \in \mathbb{R}^d$. The reconfiguration graph $\recongraph_{\mathrm{CC}}((A_1, \ldots, A_n), x)$ has vertex set consisting of all ordered tuples $(a_1, \ldots, a_n)$ with the properties that $a_i \in A_i$ for all $i$ and $x \in \convex(\{a_1, \ldots, a_n\})$. Two such tuples $(a_1, \ldots, a_n)$ and $(b_1, \ldots, b_n)$ are joined by an edge in the graph whenever there exists a unique index $j \in [n]$ such that $a_j \neq b_j$ and $x \in \convex(\{a_i : i \neq j\})$. The colorful Carath\'eodory theorem is equivalent to the statement that $\recongraph_{\mathrm{CC}}((A_1, \ldots, A_n), x)$ is nonempty whenever $x \in \convex(A_i)$ for all $i \in [n]$ and $n \ge d+1$. We show the following.

\begin{thm} \label{thm:colorful-caratheodory-reconfigure}
Let $A_1, \ldots, A_{n}$ be finite sets of points in $\mathbb{R}^d$, and let $x \in \mathbb{R}^d$. If $x \in \convex(A_i)$ for all $i \in [n]$ and $n \ge d+2$, then $\recongraph_{\mathrm{CC}}((A_1, \ldots, A_{n}), x)$ is connected.
\end{thm}

A reconfiguration version of the ordinary Helly theorem is obtained by taking all the convex set families $\cF_i$ to be the same in Theorem \ref{thm:colorful-helly-reconfigure}, and likewise a reconfiguration version of the ordinary Carath\'eodory theorem is obtained by taking all the point sets $A_i$ to be the same in Theorem \ref{thm:colorful-caratheodory-reconfigure}. One may state these latter results more intuitively as reconfigurations of sets rather than ordered tuples.

We remark that our reconfiguration version of the (non-colorful) Carath\'eodory theorem was already known by Matou\v{s}ek \cite{matousek2013lectures}. In his monograph \cite[Exercise 8.2.1]{matousek2013lectures}, he gave it as an exercise to prove the following equivalent statement: If $S$ and $T$ are $(d+1)$-element subsets of $\mathbb{R}^d$ each containing the point $x$ in their convex hulls, then there is a finite sequence $S = S_0, S_1, \ldots, S_N = T$ of $(d+1)$-element subsets $S_i$ of $S \cup T$ with each $S_i$ containing $x$ in its convex hull and $S_{i+1}$ obtained from $S_{i}$ by removing one point and adding another. The intended solution appears to have been to apply the Gale transform, so that such a sequence of sets can be translated roughly to a walk between two facets on a corresponding polytope. Thus, we are providing a topological proof of his exercise that also extends to the colorful setting. 

\bigskip

We proceed to the proofs of Theorems \ref{thm:colorful-helly-reconfigure} and \ref{thm:colorful-caratheodory-reconfigure}. We follow the approach of Kalai and Meshulam \cite{kalai2005topological}, who used the topological Hall theorem to prove a topological colorful Helly theorem about the intersection of $d$-Leray complexes and matroids. (See also \cite{cho2025colorful, holmsen2016intersection} for further topological refinements of the colorful Carath\'eodory theorem.)

For some terminology, a complex $\cC$ on ground set $V$ is said to be \textit{$d$-Leray} if the reduced homology group $\widetilde{H}_i(\cC[X])$, with rational coefficients, vanishes for all $X \subseteq V$ and $i \ge d$. For a finite family $\cF$ of sets, recall that the \textit{nerve} $\nerve(\cF)$ of $\cF$ is the simplicial complex whose vertices are all the sets in $\cF$, and a subfamily $\cG \subseteq \cF$ forms a simplex if $\bigcap \cG \neq \emptyset$. A complex is said to be \textit{$d$-representable} if it is isomorphic to the nerve of a finite family of convex sets in $\mathbb{R}^d$. Nerve theorems (e.g., Theorem \ref{lem:nerve} or \cite[Theorem 10.7]{bjorner1995topological}) imply that all $d$-representable complexes are $d$-Leray.

For a complex $\cC$ on ground set $V$, its \textit{Alexander dual} is the complex $\cC^{\star} \coloneqq \{X \subseteq V : V - X \notin \cC \}$ on the same ground set $V$. The combinatorial Alexander duality theorem states that $\widetilde{H}_{i}(\cC^{\star}) \cong \widetilde{H}_{|V| - i - 3}(\cC)$ for all $-1 \le i \le |V| - 2$ (assuming that $V \notin \cC$ and that we are using field coefficients, so that reduced homology and reduced cohomology groups are isomorphic). We only need the following corollary of this theorem.

\begin{lem}[\cite{kalai2005topological}] \label{lem:Alexander-dual}
If $\cC$ is a $d$-Leray complex on ground set $V$ with $V \notin \cC$, then $\eta_H(\cC^{\star}[X]) \ge |X| - d - 1$ for all $X \subseteq V$.
\end{lem}

The following is our reconfiguration version of the topological colorful Helly theorem by Kalai and Meshulam \cite{kalai2005topological}. Recall that for a complex $\cC$ and matroid $\cM$ on the same ground set, the reconfiguration graph $\recongraph(\cC, \cM)$ has vertex set consisting of all bases of $\cM$ that are simplices of $\cC$, and two such bases are joined by an edge if their union is a simplex of $\cC$ of cardinality $r(\cM)+1$. By taking complements of its vertices, we may view the graph $\recongraph(\cC^{\star}, \cM^{\ast})$ isomorphically as the graph whose vertex set consists of all bases $B$ of $\cM$ that are not faces of $\cC$, and two such bases $B$ and $B'$ are joined by an edge whenever $|B \cap B'| = r(\cM) - 1$ and $B \cap B' \notin \cC$. 

\begin{thm} \label{thm:topological-Helly-reconfigure}
Let $\cC$ be a complex and let $\cM$ be a matroid on the same ground set $V$. If $\cC$ is $d$-Leray and $r(\cM[V - A]) \ge d+2$ for all $A \in \cC$, then $\recongraph(\cC^{\star}, \cM^{\ast})$ is connected.
\end{thm}

\begin{proof}
We apply Theorem \ref{thm:complex-matroid-reconfiguration}, wishing to show that $\eta_H(\cC^{\star}[X]) + r(\cM^{\ast}[V - X]) \ge r(\cM^{\ast})+1$ for all nonempty $X \subseteq V$. If $X \in \cC^{\star}$, then $\cC^{\star}[X]$ is a simplex and thus $\eta_H(\cC^{\star}[X]) = \infty$. On the other hand, if $X \notin \cC^{\star}$, then $V - X \in \cC$, so applying Lemma \ref{lem:Alexander-dual} and the theorem assumption with $A \coloneqq V - X$, we get that
\begin{align*}
	\eta_H(\cC^{\star}[X]) + r(\cM^{\ast}[V - X]) &\ge (|X| - d - 1) + (|V - X| - r(\cM) + r(\cM[X])) \\
	&= r(\cM^{\ast}) + r(\cM[V - A]) - d - 1 \\
	&\ge r(\cM^{\ast}) + 1,
\end{align*}
as required.
\end{proof}

For comparison, Kalai and Meshulam \cite{kalai2005topological} proved the existence of a basis $B$ of $\cM$ that is not a face of $\cC$ under the assumptions that $\cC$ is $d$-Leray and $r(\cM[V - A]) \ge d+1$ for all $A \in \cC$. Now we finish the proofs of Theorems \ref{thm:colorful-helly-reconfigure} and \ref{thm:colorful-caratheodory-reconfigure}.

\begin{proof}[Proof of Theorem \ref{thm:colorful-helly-reconfigure}]
Let $\cN \coloneq \nerve(\cF)$ be the nerve of the family $\cF \coloneqq \cF_1 \cup \cdots \cup \cF_{n}$. Let $\cC$ be the complex on vertex set $V \coloneqq \{(C,i) : i \in [n], C \in \cF_i \}$, where a subset $A \subseteq V$ lies in $\cC$ whenever $\pi(A) \in \cN$. Here, $\pi : (C,i) \mapsto C$ denotes the natural projection, which gives a homotopy equivalence between $\cC$ and $\cN$. Let $\cM$ be the partition matroid on $V$ defined by the partition $\cF_1 \times \{1\}, \ldots, \cF_{n} \times \{n\}$. Then, from the complementary point of view described above, the graphs $\recongraph(\cC^{\star}, \cM^{\ast})$ and $\recongraph_{\mathrm{CH}}(\cF_1, \ldots, \cF_{n})$ are isomorphic. The nerve $\cN$ is $d$-representable by definition, and hence it is $d$-Leray. Thus the complex $\cC$ is also $d$-Leray. In addition, using that $n \ge d+2$, the condition $\bigcap \cF_i = \emptyset$ for all $i \in [n]$ implies that $r(\cM[V - A]) \ge d+2$ for all $A \in \cC$. Applying Theorem \ref{thm:topological-Helly-reconfigure}, we deduce that $\recongraph_{\mathrm{CH}}(\cF_1, \ldots, \cF_{n})$ is connected.
\end{proof}

\begin{proof}[Proof of Theorem \ref{thm:colorful-caratheodory-reconfigure}]
For all $i \in [n]$, assume that $A_i \coloneqq \{a_{i,1}, \ldots, a_{i,k_i}\} \subset \mathbb{R}^d$ with $k_i \ge 1$. Consider the families $\cF_i \coloneqq \{H_{i,j} : j \in [k_i]\}$ consisting of the (possibly empty) closed half-spaces $H_{i,j} \coloneqq \{y \in \mathbb{R}^d : (a_{i,j} - x)^\top y \ge 1\}$, for $i \in [n]$. By an application of Farkas' lemma, for all subsets $A \subseteq \bigcup_{i=1}^n A_i$, say $A \coloneqq \{a_{i_1,j_1}, \ldots, a_{i_{\ell},j_{\ell}}\}$, we have $x \in \convex(A)$ if and only if $\bigcap_{k=1}^{\ell} H_{i_{k},j_{k}} = \emptyset$. In particular, the reconfiguration graphs $\recongraph_{\mathrm{CC}}((A_1, \ldots, A_{n}), x)$ and $\recongraph_{\mathrm{CH}}(\cF_1, \ldots, \cF_{n})$ are isomorphic. The condition $x \in \convex(A_i)$ is equivalent to the condition $\bigcap \cF_i = \emptyset$. Applying Theorem \ref{thm:colorful-helly-reconfigure}, we deduce that $\recongraph_{\mathrm{CC}}((A_1, \ldots, A_{n}), x)$ is connected.
\end{proof}

\subsection{Reconfigurations for Tverberg's theorem}

In this subsection, we prove our reconfiguration version of Tverberg's theorem \cite{tverberg1966generalization} which we stated in the Introduction (Theorem \ref{thm:Tverberg-reconfiguration-1}). We repeat the theorem statement below. 

\begin{thm} \label{thm:tverberg-reconfiguration}
	Let $X$ be a finite set of points in $\mathbb{R}^d$, and let $r \ge 1$ be an integer. If $|X| \ge (d+1)(r-1)+2$, then $\recongraph_{\mathrm{Tv}}(X, r)$ is connected.
\end{thm}

Our proof of Theorem \ref{thm:tverberg-reconfiguration} applies Sarkaria's \cite{sarkaria1992tverberg} well-known transformation relating Tverberg's theorem to the colorful Carath\'eodory theorem. We use the simplified tensor form of the transformation due to B\'ar\'any and Onn \cite{barany1997colourful}. Say that $X \coloneqq \{x_1, \ldots, x_n\} \subset \mathbb{R}^d$. Let $w_1, \ldots, w_r$ be vectors in $\mathbb{R}^{r-1}$ where the unique linear dependence, up to scaling, is $w_1 + \cdots + w_r = 0$. For $i \in [n]$ and $j \in [r]$, we define the tensor
\begin{align*}
	\overline{x}_{i,j} \coloneqq \begin{pmatrix}x_i\\1\end{pmatrix} \otimes w_j \in \mathbb{R}^{(d+1)(r-1)}.
\end{align*}
One may view the tensor product $y \otimes z$ of nonzero vectors $y = (y_1, \ldots, y_{d+1})$ and $z = (z_1, \ldots, z_{r-1})$ as the $(d+1) \times (r-1)$ rank-one matrix $(y_k \cdot z_{\ell})_{k,\ell}$. 

Allowing parts to be empty, there is a natural bijection between ordered partitions $(X_1, \ldots, X_r)$ of $X$ into $r$ parts and functions $j : [n] \to [r]$, where $n = |X|$. Specifically, we associate the function $j : [n] \to [r]$ with the ordered partition $P_j \coloneqq (X_{1,j}, \ldots, X_{r,j})$ where $X_{k,j} \coloneqq \{x_i \in X : j(i) = k\}$. We also associate the function $j : [n] \to [r]$ with the tuple $T_j = (\overline{x}_{1,j(1)}, \ldots, \overline{x}_{n,j(n)})$ consisting of the tensors $\overline{x}_{i,j(i)}$ defined above. The following fundamental lemma of Sarkaria \cite{sarkaria1992tverberg} relates Tverberg's theorem to the colorful Carath\'eodory theorem, proven in the form below in \cite[Lemma 2]{arocha2009very}.

\begin{lem}[\cite{arocha2009very, sarkaria1992tverberg}] \label{lem:Sarkaria}
The ordered partition $P_j \coloneqq (X_{1,j}, \ldots, X_{r,j})$ satisfies $\bigcap_{k=1}^r \convex(X_{k,j}) \neq \emptyset$ if and only if the tuple $T_j \coloneqq (\overline{x}_{1,j(1)}, \ldots, \overline{x}_{n,j(n)})$ satisfies $0 \in \convex(\{\overline{x}_{i,j(i)} : i \in [n] \})$.
\end{lem}

Using Lemma \ref{lem:Sarkaria}, we may now finish the proof of Theorem \ref{thm:tverberg-reconfiguration}.

\begin{proof}[Proof of Theorem \ref{thm:tverberg-reconfiguration}]
Let $X \coloneqq \{x_1, \ldots, x_n\}$ where $n \ge (d+1)(r-1)+2$. For all $i \in [n]$, consider the set of tensors $A_i \coloneqq \{\overline{x}_{i,1}, \ldots, \overline{x}_{i,r}\} \subset \mathbb{R}^{(d+1)(r-1)}$ as defined above. Each of the sets $A_i$ contains the origin $0$ in its convex hull, as the sum of the elements of $A_i$ is $0$. By Theorem \ref{thm:colorful-caratheodory-reconfigure}, the reconfiguration graph $\recongraph_{\mathrm{CC}}((A_1, \ldots, A_n), 0)$ is connected. It is then easy to verify using Lemma \ref{lem:Sarkaria} that $\recongraph_{\mathrm{Tv}}(X, r)$ is isomorphic to $\recongraph_{\mathrm{CC}}((A_1, \ldots, A_n), 0)$, and the theorem follows.
\end{proof}

The case $r = 2$ of Theorem \ref{thm:tverberg-reconfiguration} provides a reconfiguration version of Radon's theorem. Oliveros, Rold\'an, Sober\'on, and Torres \cite{oliveros2025tverberg} proved a different reconfiguration version of Radon's theorem. For a finite point set $X \subset \mathbb{R}^d$, an \textit{unordered Radon partition} is a partition $\{X_1, X_2\}$ of $X$ satisfying $\convex(X_1) \cap \convex(X_2) \neq \emptyset$. The above authors showed that for any two unordered Radon partitions $P, Q$ of $X$, if they exist, then there exists a finite sequence $P = P_0, P_1, \ldots, P_N = Q$ of unordered Radon partitions of $X$ in which $P_{i+1}$ is obtained from $P_i$ by moving one point $x$ from one part of $P_i$ to the other. Their argument was a linear interpolation between affine dependencies associated with the unordered Radon partitions $P$ and $Q$. Here, we adapt their approach to give an alternative, non-topological proof of the case $r = 2$ of Theorem \ref{thm:tverberg-reconfiguration}.

\begin{proof}[Alternative proof of Theorem \ref{thm:tverberg-reconfiguration} when $r=2$]
Consider any two distinct ordered Radon partitions $(X_1, X_2)$ and $(Y_1, Y_2)$ of $X \coloneqq \{x_1, \ldots, x_n\}$, where $|X| = n \ge d+3$. Then there exist $n$-tuples $(\alpha_1, \ldots, \alpha_n)$ and $(\beta_1, \ldots, \beta_n)$ in $\mathbb{R}^d$, with entries not all $0$ and having the properties that 
\begin{itemize}
	\item $\sum_{i = 1}^n \alpha_i = \sum_{i = 1}^n \beta_i = 0$ and $\sum_{i=1}^n \alpha_i x_i = \sum_{i=1}^n \beta_i x_i = 0$, and
	\item $\alpha_i \ge 0$ whenever $x_i \in X_1$, $\alpha_i \le 0$ whenever $x_i \in X_2$, $\beta_i \ge 0$ whenever $x_i \in Y_1$, and $\beta_i \le 0$ whenever $x_i \in Y_2$.
\end{itemize}
Consider the line segment $L(t) \coloneqq (1 - t) \cdot (\alpha_1, \ldots, \alpha_n) + t \cdot (\beta_1, \ldots, \beta_n)$ for $t \in [0,1]$. Assume first that $L(t) \neq 0$ for all $t \in [0,1]$. We define a reconfiguration sequence from $(X_1, X_2)$ to $(Y_1, Y_2)$ as follows: Increasing $t$ continuously from $0$ to $1$, every time the value at an entry $i$ of $L(t)$ changes sign, we move $x_i$ from one part to the other. If multiple entries become $0$ simultaneously, we do these moves in arbitrary order. Using that $L(t)$ is never $0$, each such element swap is a valid move in the reconfiguration graph. Thus, we use this to derive a valid reconfiguration sequence from $(X_1, X_2)$ to $(Y_1, Y_2)$. 

Now assume that $L(t_0) = 0$ for some $t_0 \in (0,1)$. Let $i \in [n]$ be such that $\alpha_i > 0$, so that $\beta_i < 0$. Using that $|X| \ge d+3$, we apply Radon's theorem to get an ordered Radon partition $(Z_1, Z_2)$ of $X - \{x_i\}$. Let $(\gamma_1, \ldots, \gamma_n)$ be an $n$-tuple in $\mathbb{R}^n$ associated with the Radon partition $(Z_1 \cup \{x_i\}, Z_2)$ of $X$, where $\gamma_i = 0$. Consider the two line segments $L_1(t) \coloneqq (1 - t) \cdot (\alpha_1, \ldots, \alpha_n) + t \cdot (\gamma_1, \ldots, \gamma_n)$ and $L_2(t) \coloneqq (1 - t) \cdot (\gamma_1, \ldots, \gamma_n) + t \cdot (\beta_1, \ldots, \beta_n)$ for $t \in [0,1]$. Since $\alpha_i > 0$, $\beta_i < 0$, and $\gamma_i = 0$, we have $L_1(t), L_2(t) \neq 0$ for all $t \in [0,1]$. Therefore, applying the same swapping procedure above, we get a valid reconfiguration sequence from $(X_1, X_2)$ to $(Z_1 \cup \{x_i\}, Z_2)$, followed by a valid reconfiguration sequence from $(Z_1 \cup \{x_i\}, Z_2)$ to $(Y_1, Y_2)$. This finishes the proof.
\end{proof}

\subsection{Higher dimensional connectedness}

In this subsection, we state our higher dimensional homological connectedness versions of the colorful Helly, colorful Carath\'eodory, and Tverberg theorems. These follow the same proofs as Theorems \ref{thm:colorful-helly-reconfigure}, \ref{thm:colorful-caratheodory-reconfigure}, and \ref{thm:tverberg-reconfiguration}, but use Theorem \ref{thm:complex-matroid-connectedness} in place of Theorem \ref{thm:complex-matroid-reconfiguration}. The proofs work with homology coefficients lying in any fixed field. In turn, the results also hold with homology coefficients lying in any fixed abelian group, via the universal coefficient theorem.

Recall that for a complex $\cC$ and matroid $\cM$ on the same ground set $V$, the intersection complex $\intersectioncomplex(\cC, \cM)$ is the order complex on the collection of faces of $\cC$ that contain a basis of $\cM$, partially ordered by inclusion. Letting $\cC^{\star}$ denote the Alexander dual of $\cC$ and $\cM^{\ast}$ denote the matroid dual of $\cM$, we may view the complex $\intersectioncomplex(\cC^{\star}, \cM^{\ast})$ isomorphically as the order complex on the collection of independent sets of $\cM$ that are not faces of $\cC$, partially ordered by inclusion. We have the following generalization of the Kalai--Meshulam theorem \cite{kalai2005topological} and of Theorem \ref{thm:topological-Helly-reconfigure}. (Note that definition of $d$-Leray should be adjusted according to the choice of homology coefficients.)

\begin{thm}
Let $\cC$ be a complex and let $\cM$ be a matroid on the same ground set $V$, and let $d, m \ge 0$ be integers. If $\cC$ is $d$-Leray and $r(\cM[V - A]) \ge d+m+1$ for all $A \in \cC$, then $\intersectioncomplex(\cC^{\star}, \cM^{\ast})$ is homologically $(m-1)$-connected.
\end{thm}

For a higher dimensional connectedness version of the colorful Helly theorem \cite{barany1982generalization}, we are given finite families $\cF_1, \ldots \cF_n$ each consisting of convex subsets of $\mathbb{R}^d$. We let $\star$ be a symbol that indicates the absence of a choice of representative. We define the complex $\textbf{ColHel}(\cF_1, \ldots, \cF_n)$ as the order complex on the collection of tuples of the form $(C_1, \ldots, C_n)$ satisfying $C_i \in \cF_i \cup \{\star\}$ for all $i \in [n]$ and $\bigcap_{i \in [n] : C_i \neq \star} C_i = \emptyset$. The partial order $\le$ on this collection defined by $(C_1, \ldots, C_n) \le (D_1, \ldots, D_n)$ whenever $C_i \in \{D_i, \star\}$ for all $i \in [n]$. We have the following generalization of Theorems \ref{thm:colorful-helly} and \ref{thm:colorful-helly-reconfigure} .

\begin{thm}
Let $\cF_1, \ldots, \cF_{n}$ be finite families of convex subsets of $\mathbb{R}^d$, and let $m \ge 0$ be an integer. If $\bigcap \cF_i = \emptyset$ for all $i \in [n]$ and $n \ge d+m+1$, then $\mathbf{ColHel}(\cF_1, \ldots, \cF_{n})$ is homologically $(m-1)$-connected.
\end{thm}

For a higher dimensional connectedness version of the colorful Carath\'eodory theorem \cite{barany1982generalization}, we are given finite sets $A_1, \ldots, A_n$ of points in $\mathbb{R}^d$ and a point $x \in \mathbb{R}^d$. We define the complex $\mathbf{ColCat}((A_1, \ldots, A_n), x)$ as the order complex on the collection of tuples of the form $(a_1, \ldots, a_n)$ satisfying $a_i \in A_i \cup \{\star\}$ for all $i \in [n]$ and $x \in \convex(\{a_i : i \in [n], a_i \neq \star\})$. The partial order $\le$ on this collection is defined by $(a_1, \ldots, a_n) \le (b_1, \ldots, b_n)$ whenever $a_i \in \{b_i, \star\}$ for all $i \in [n]$. We have the following generalization of Theorems \ref{thm:colorful-caratheodory} and \ref{thm:colorful-caratheodory-reconfigure}.

\begin{thm}
Let $A_1, \ldots, A_{n}$ be finite sets of points in $\mathbb{R}^d$, let $x \in \mathbb{R}^d$, and let $m \ge 0$ be an integer. If $x \in \convex(A_i)$ for all $i \in [n]$ and $n \ge d+m+1$, then $\mathbf{ColCat}((A_1, \ldots, A_{n}), x)$ is homologically $(m-1)$-connected.
\end{thm}

For a higher dimensional connectedness version of Tverberg's theorem \cite{tverberg1966generalization}, we are given a finite point set $X \subset \mathbb{R}^d$ and an integer $r \ge 1$. We define the complex $\mathbf{Tver}(X, r)$ as the order complex on the collection of all ordered tuples of the form $(X_1, \ldots, X_r)$ such that $X_i \subseteq X$ for all $i \in [r]$, the $X_i$'s are pairwise disjoint, and $\bigcap_{i = 1}^r \convex(X_i) \neq \emptyset$. The partial order $\le$ on this collection is defined by $(X_1, \ldots, X_r) \le (Y_1, \ldots, Y_r)$ whenever $X_i \subseteq Y_i$ for all $i \in [r]$. We have the following generalization of Theorems \ref{thm:tverberg} and \ref{thm:tverberg-reconfiguration}.

\begin{thm} \label{thm:tverberg-homological-connectedness}
Let $X$ be a finite set of points in $\mathbb{R}^d$, and let $r \ge 1$ and $m \ge 0$ be integers. If $|X| \ge (d+1)(r - 1) + m + 1$, then $\mathbf{Tver}(X, r)$ is homologically $(m-1)$-connected.
\end{thm}

\section{Conclusion}

In this paper, we demonstrated how topological methods developed around Hall's transversal theorem naturally adapt to showing that reconfiguration graphs are connected in various natural combinatorial settings. Our main reconfiguration theorem (Theorem \ref{thm:topological-reconfiguration-graph}) was a variation of the usual topological Hall theorem (Theorem \ref{thm:topological-hall}), which itself was a useful generalization of Hall's transversal theorem as well as Rado's matroid transversal theorem. We extended our reconfiguration theorem to a more general theorem (Theorem \ref{thm:topological-reconfiguration-complex}) about the higher order homological connectedness of the space of colorful simplices, in the form of the \textit{colorful complex}. We then further generalized to a matroidal setting, proving a result (Theorem \ref{thm:complex-matroid-reconfiguration}) about the homological connectedness of the space of sets of given size lying in the intersection of a complex and a matroid, in the form of the \textit{intersection complex}. We gave applications focusing on natural reconfiguration problems in graph theory and discrete geometry, emphasizing applications that have found previous success with the usual topological Hall theorem. We gave a complete alternative proof of Theorem \ref{thm:BKO} of Buys, Kang, and Ozeki \cite{buys2025reconfiguration} about reconfigurations of independent transversals in graphs, and we confirmed a conjecture of Oliveros, Rold{\'a}n, Sober{\'o}n, and Torres \cite[Conjecture 2]{oliveros2025tverberg} about a reconfiguration version of Tverberg's theorem \cite{tverberg1966generalization} (Theorem \ref{thm:Tverberg-reconfiguration-1}). Topological perspectives on reconfigurations have previously found use in the world of graph homomorphism reconfigurations (e.g., \cite{dochtermann2009hom, dochtermann2023homomorphism, wrochna2020homomorphism}), with a somewhat different flavor, and we believe that this paper's topological perspective and combinatorial contributions will also be fruitful in future directions on reconfigurations of combinatorial structures, as well as studies of higher dimensional topological connectedness of solution spaces.

We described some open problems in previous sections. Two of these problems were about reconfiguration versions of asymptotic existence results on equi-$n$-squares and Latin $n$-squares (Problems \ref{prob:equi} and \ref{prob:latin}). Another one was a conjectured reconfiguration version of Ryser's conjecture (Conjecture \ref{conj:ryser}). Finally, two others were about reconfigurations of list colorings in color degree settings (Problems \ref{prob:vertex-coloring} and \ref{prob:edge-coloring}). It is also natural to ask about the higher dimensional topological connectedness of the associated configuration spaces. We conclude this paper by highlighting a few broader problems. 

A significant limitation with our topological approach, similar to the usual topological Hall theorem, is that it does not come with an efficient algorithm for finding reconfiguration sequences. Indeed, it does not even provide a bound on the diameter of the reconfiguration graph. It is easy to construct examples where our reconfiguration topological Hall theorem applies but where some reconfiguration sequences require using all the colorful simplices of the complex. This is not ideal in most combinatorial applications, where there are often exponentially many configurations but reconfiguration graphs are usually thought to have polynomially-sized diameters. In addition, our work does not yet say anything about the mixing times of the Markov chains associated with our reconfiguration graphs, whereas this is a problem of significant interest for obtaining efficient approximate random sampling and counting algorithms on the objects of interest (see \cite{jerrum1995very} for more on this topic). These remarks motivate the following problems.

\begin{prob}
For the combinatorial applications that we discussed, investigate the complexity of finding reconfiguration sequences between two input configurations.
\end{prob}

\begin{prob}
For those same combinatorial settings, find effective bounds on the diameter of the reconfiguration graphs.
\end{prob}

\begin{prob}
Investigate mixing times of Markov chains associated with our reconfiguration graphs, in particular determining when they are slowly or rapidly mixing.
\end{prob}

On the more topological side, our generalized topological Hall theorem (Theorem \ref{thm:topological-reconfiguration-complex}), about the topological connectedness of the colorful complex, was done only in the homological setting $\eta_H$. Naturally, here we ask whether the same result also holds in the homotopical setting $\eta_{\pi}$, similar to the homotopical proof we presented in the case $m = 1$ in Section \ref{sec:proofs}.

\begin{prob}
Determine whether Theorem \ref{thm:topological-reconfiguration-complex} also holds when homological connectedness $\eta_H$ is replaced by homotopical connectedness $\eta_{\pi}$.
\end{prob}

In particular, it would be interesting to determine when the complexes discussed in this paper are simply-connected. Going further, it is worth investigating the homotopy types of these complexes.

\section*{Acknowledgments}

The author would like to thank Penny Haxell for stimulating discussions. This research was supported in full by the Austrian Science Fund (FWF) [10.55776/F1002]. For open access purposes, the author has applied a CC BY public copyright license to any author accepted manuscript version arising from this submission.

\printbibliography

\end{document}